\documentclass[12pt,fullpage]{article}

\usepackage{amssymb}
\usepackage{amsmath}
\usepackage{cases}
\usepackage{amsfonts}
\usepackage{amsmath,amssymb}
\usepackage{multicol}
\usepackage{color}
\usepackage{graphicx}
\usepackage{graphics}
\usepackage{mathrsfs}
\def\R{\mathbb{R}}

\def\epsilon{\varepsilon}

\newcommand{\be}{\begin{equation}}
\newcommand{\ee}{\end{equation}}
\newcommand{\baa}{\begin{array}}
\newcommand{\eaa}{\end{array}}
\newcommand{\ba}{\begin{eqnarray}}
\newcommand{\ea}{\end{eqnarray}}

\newtheorem{theorem}{Theorem}[section]
\newtheorem{lemma}[theorem]{Lemma}

\newtheorem{definition}[theorem]{Definition}
\newtheorem{remark}[theorem]{Remark}

\newtheorem{proposition}[theorem]{Proposition}
\numberwithin{equation}{section}
\newenvironment{proof}[1][Proof]{\noindent\textbf{#1.} }{\hfill $\Box$}
\allowdisplaybreaks

\begin{document}
\date{}
\title{\bf{Pushed fronts of monostable reaction-diffusion-advection equations}\thanks{The author is supported by the fundamental research funds for the central universities and  NSF of China (no. 12101456).}}
\author{Hongjun Guo \textsuperscript{a} \\
\\
\footnotesize{\textsuperscript{a} School of Mathematical Sciences, Institute for Advanced Study, Tongji University, Shanghai, China}}
\maketitle

\begin{abstract}\noindent{In this paper, we prove some qualitative properties of pushed fronts for the periodic reaction-diffusion-equation with general monostable nonlinearities. Especially, we prove the exponential behavior of pushed fronts when they are approaching their unstable state. The proof also allows us to get the exponential behavior of pulsating fronts with speed $c$ larger than the minimal speed. Through  the exponential behavior, we  finally prove the stability of pushed fronts.}
\noindent{}
\vskip 0.1cm
\end{abstract}


\section{Introduction}

This paper follows up the articles of Hamel \cite{H2008} and his with Roques \cite{HR2011} on qualitative properties of pulsating fronts in the periodic monostable case. In \cite{H2008,HR2011}, they have proven various properties for pulsating fronts under the KPP assumption. However, the case without the KPP assumption and when the minimal speed of the fronts is larger than $c_0$ where $c_0$ is defined later by \eqref{c-star}, was not covered by \cite{H2008,HR2011}. In their articles \cite{H2008,HR2011}, they have proven the monotonicity, exponential behavior, uniqueness and stability of KPP pulsating fronts. Here, we show that these properties still hold for the non-KPP case.

We consider the following reaction-diffusion-advection equation
\begin{eqnarray}\label{RD}
\left\{\begin{array}{lll}
u_t-\nabla \cdot (A(z)\nabla u)+q(z)\cdot \nabla u =f(z,u), &&t\in\R, z=(x,y)\in \overline{\Omega},\\
\nu A\nabla u=0, &&t\in\R, z\in\partial\Omega,
\end{array}
\right.
\end{eqnarray}
where $\Omega$ is a smooth unbounded domain of $\R^{N}$ and $\nu$ denotes the outward unit normal on the boundary $\partial\Omega$. Assume that $x$ is $d$-dimensional and $y$ is $(N-d)$-dimensional, where $d$ is an integer of $\{1,\ \cdots,\ N\}$. Denote
$$x=(x_1,\cdots,x_d),\quad y=(x_{d+1},\cdots,x_{N}),\quad z=(x,y).$$
The underlying domain $\Omega$ is assumed to be of class $C^{2,\alpha}$ for some $\alpha>0$, periodic with respect to $x$ and bounded with respect to $y$. That is, there are $d$ positive real numbers $L_1$, $\cdots$, $L_d$ such that
\begin{eqnarray*}
\left\{\begin{array}{lll}
\exists R\ge 0,\ \forall z=(x,y)\in \Omega,\ |y|\le R,\\
\forall k\in L_1\mathbb{Z}\times\cdots\times L_d\mathbb{Z}\times \{0\}^{N-d},\ \Omega=\Omega+k.
\end{array}
\right.
\end{eqnarray*}
We denote the periodicity cell by 
$$\mathcal{C}=\{(x,y)\in \Omega; \ x\in (0,L_1)\times\cdots\times (0,L_d)\}.$$
Typical examples of such domains are the whole space $\R^N$, the whole space with periodic perforations, infinite cylinders with constant or periodically undulating sections, etc. 

The coefficients $A$, $q$, $f$ of \eqref{RD} are assumed to be periodic with respect to $x$, in the sense that
$$\omega(x+k,y)=\omega(x,y) \hbox{ for all $(x,y)\in \overline{\Omega}$ and $k\in L_1\mathbb{Z}\times\cdots\times L_d\mathbb{Z}$}.$$ 
The matrix field $A(x,y)=(A_{ij}(x,y))_{1\le i,j\le N}$ is symmetric of class $C^{1,\alpha}(\overline{\Omega})$ and satisfies
\begin{eqnarray*}
\left\{\begin{array}{lll}
\exists 0<c_1\le c_2,\ \forall \xi\in \R^N,\ \forall (x,y)\in \overline{\Omega},\\
c_1 |\xi |^2\le \sum_{1\le i,j\le N} A_{ij}(x,y) \le c_2 |\xi |^2.
\end{array}
\right.
\end{eqnarray*}
The vector field $q(x,y)=(q_i(x,y))_{1\le i\le N}$ is of class $C^{0,\alpha}(\overline{\Omega})$ and divergence free $\nabla\cdot q=0$ in $\overline{\Omega}$.
The nonlinearity $f(x,y,u)$ is continuous, of class $C^{0,\alpha}$ with respect to $(x,y)$ locally uniformly in $u\in \R$ and of class $C^{1,\alpha}$ with respect to $u$ locally uniformly in $(x,y)\in \overline{\Omega}$. We assume that $0$ and $1$ are zeroes of $f$, that is,
$$f(x,y,0)=f(x,y,1)=0, \hbox{ for all $(x,y)\in \overline{\Omega}$}.$$
This means that $0$ and $1$ are stationary solutions of \eqref{RD}. Notice that $0$ and $1$ could be more general stationary solutions and they here are taken without loss of generality, see \cite{H2008} for the transformation from general stationary solutions $p(x)$ and $q(x)$ to $0$ and $1$. The general framework of this work is under monostable nonlinearities. Therefore, to clarify this, we let $\mu_0$, $\mu_1$ be the principal eigenvalues and $\psi_0(x,y)$, $\psi_1(x,y)$ be the principal eigenfunctions of 
\begin{eqnarray*}
\left\{\begin{array}{lll}
-\nabla\cdot (A(x,y)\nabla \psi_0) +q(x,y)\cdot \nabla\psi_0 -f_u(x,y,0)\psi_0=\mu_0\psi_0, \quad &&\hbox{ in $\overline{\Omega}$},\\
\nu A(x,y)\nabla\psi_0=0, \quad &&\hbox{ on $\partial\Omega$}.
\end{array}
\right.
\end{eqnarray*}
and
\begin{eqnarray}\label{mu1}
\left\{\begin{array}{lll}
-\nabla\cdot (A(x,y)\nabla \psi_1) +q(x,y)\cdot \nabla\psi_1 -f_u(x,y,1)\psi_1=\mu_1\psi_1, \quad &&\hbox{ in $\overline{\Omega}$},\\
\nu A(x,y)\nabla\psi_1=0, \quad &&\hbox{ on $\partial\Omega$}.
\end{array}
\right.
\end{eqnarray}
respectively. Throughout this paper, we assume that $0$ and $1$ are linearly unstable and linearly stable respectively, that is,
$$\mu_0<0 \hbox{ and } \mu_1>0.$$
For some properties of pulsating fronts, the assumption that $1$ is linearly stable could be more general as being weakly stable, see \cite{H2008,HR2011}. But since the main difficulty and interest lie on the part of the front approaching $0$, we only assume $1$ being linearly stable for simplicity. A typical example of such a monostable nonlinearity with $\mu_0<0$ and $\mu_1>0$ is when $f_{u}(x,y,0)>0$ and $f_u(x,y,0)<0$.

We are concerned with pulsating fronts for the periodic framework, which are defined in the following. 

\begin{definition}
Given a unit vector $e\in \R^{d}\times\{0\}^{N-d}$, a time-global classical solution $u(t,x,y)$ of \eqref{RD} is called a pulsating front connecting $0$ and $1$, propagating in the direction $e$ with speed $c$ if there exist a continuous function $\phi_c: \R\times\overline{\Omega}\rightarrow \R$ and a constant $c$ satisfying $$(x,y)\rightarrow \phi_c(s,x,y) \hbox{ is periodic in $\overline{\Omega}$ for all $s\in \R$},$$
and
$$\phi_c(-\infty,x,y)=1,\ \phi_c(+\infty,x,y)=0 \hbox{ uniformly in $(x,y)\in \overline{\Omega}$},$$
such that
$$u(t,x,y)=\phi_c(x\cdot e -ct,x,y) \hbox{ for all $(t,x,y)\in \R\times\overline{\Omega}$}.$$
\end{definition}
With a slight abuse of notion, $x\cdot e$ denotes $x_1 e_1+\cdots+x_{d} e_{d}$ where $e_1$, $\cdots$, $e_d$ are the first $d$ components of the vector $e$. We only consider the fronts between $0$ and $1$, that is,
$$0<u(t,x,y)<1 \hbox{ for all $(t,x,y)\in \R\times\overline{\Omega}$}.$$
Some existence results of pulsating fronts are known under some more restrictions on $f$ and $q$, see \cite{BH2002,BHN,BHR}. For instance, if 
\begin{eqnarray}\label{ex-result}
\left\{\begin{array}{lll}
f(x,y,u)>0 \hbox{ for all $(x,y)\in \overline{\Omega}$ and $u\in (0,1)$},\\
f(x,y,u) \hbox{ is nonincreasing w.r.t. $u$ in a left neighborhood of $1$},\\
\nabla\cdot q=0 \hbox{ in $\overline{\Omega}$},\ q\cdot \nu=0 \hbox{ on $\partial \Omega$},\\
\int_{\mathcal{C}} q_i(x,y)dxdy=0 \hbox{ for $1\le i\le d$},
\end{array}
\right.
\end{eqnarray}
and it holds the KPP assumption
\be\label{KPP}
f(x,y,u)\le f_{u}(x,y,0)u \hbox{ for all $(x,y)\in\overline{\Omega}$ and $u\in [0,1]$},
\ee
then, given any unit vector $e\in\R^d\times\{0\}^{N-d}$, there is a minimal speed $c_*(e)>0$ such that pulsating fronts $\phi_c(x\cdot e-ct,x,y)$ exist if and only if $c\ge c_*(e)$, see \cite{BHN}.

We give some notions here. For each $\lambda\in \R$, call $k(\lambda)$ the principal eigenvalue of the operator
$$L_{\lambda} \varphi:=-\nabla\cdot (A\nabla\varphi) +2\lambda eA\nabla \varphi +q\cdot \nabla\varphi +[\lambda \nabla\cdot (A e) -\lambda q\cdot e -\lambda^2 eAe -f_u(x,y,0)]\varphi,$$
on the set
$$E_{\lambda}=\{\varphi\in C^2(\overline{\Omega});\ \varphi \hbox{ is periodic in $\overline{\Omega}$ and $\nu A\nabla\varphi=\lambda (\nu A e)\varphi$ on $\partial\Omega$}\}.$$
Notice that $k(0)=\mu_0<0$.
Define
\be\label{c-star}
c_0(e)=\inf_{\lambda>0} \Big(-\frac{k(\lambda)}{\lambda}\Big),
\ee
and
$$c^{*}(e)=\inf\{c\in\R; \hbox{ $c$ is the speed of a pulsating front $\phi(x\cdot e -ct,x,y)$ of \eqref{RD}}\}.$$
By Proposition~1.2 of \cite{H2008}, we know that $c^*(e)\ge c_0(e)$. In the sequel, we will assume that there is a pulsating front with speed $c^*(e)$. 

If \eqref{ex-result} and the KPP assumption \eqref{KPP} hold, we know that there is a pulsating front $\phi_{c_*}(x\cdot e-c^*(e)t,x,y)$ with $c^*(e)=c_0(e)$. We recall a few results for the qualitative properties of pulsating fronts under the KPP assumption \eqref{KPP} from \cite{H2008}. The monotonicity of all pulsating fronts (with and without the KPP assumption \eqref{KPP}) has been proved in \cite{H2008}, that is, $\phi_s(s,x,y)<0$ for all $(s,x,y)\in \R\times\overline{\Omega}$. The exponential behavior of $\phi(s,x,y)$ approaching $0$ with the KPP assumption \eqref{KPP}
has also been proved in \cite{H2008}. We summarize these results as following: under \eqref{KPP}, 
\begin{itemize}
\item[(i)] if there is a pulsating front $\phi_c(x\cdot e-c t,x,y)$ with $c>c_0(e)$, then there is $B>0$ such that
$$\phi_c(s,x,y)\sim B e^{-\lambda_c s}\psi_{\lambda_c}(x,y) \hbox{ as $s\rightarrow +\infty$ uniformly in $(x,y)\in \overline{\Omega}$},\footnote{For functions $f(s,x,y)$ and $g(s,x,y)$, the relationship ``$f\sim g$ as $s\rightarrow +\infty$ uniformly in $(x,y)\in\overline{\Omega}$" is understood as $\sup_{(x,y)\in\overline{\Omega}, s\ge \tau} |f/g-1|\rightarrow 0$ as $\tau\rightarrow +\infty$.}.$$
where $\lambda_c$ is the minimal root of $k(\lambda)+c\lambda=0$ and $\psi_{\lambda_c}$ is the principle eigenfunction of $L_{\lambda_c}\psi=k(\lambda_c)\psi$ with $\|\psi_{\lambda_c}\|_{L^{\infty}(\mathcal{C})}=1$,

\item[(ii)] if there is a pulsating front $\phi_{c_0}(x\cdot e-c_0(e) t,x,y)$, then there is $B>0$ such that
$$\phi_{c_0}(s,x,y)\sim B s^{2m+1} e^{-\lambda_{c_0} s}\psi_{\lambda_{c_0}}(x,y) \hbox{ as $s\rightarrow +\infty$ uniformly in $(x,y)\in \overline{\Omega}$},$$
where $\lambda_{c_0}$ is the unique root of $k(\lambda)+c_0\lambda=0$, $2m+2$ is its multiplicity and $\psi_{\lambda_{c_0}}$ is the principle eigenfunction of $L_{\lambda_{c_0}}\psi=k(\lambda_{c_0})\psi$ with $\|\psi_{\lambda_{c_0}}\|_{L^{\infty}(\mathcal{C})}=1$.
\end{itemize}
For the general assumption, that is, without the KPP assumption \eqref{KPP}, Hamel \cite{H2008} also got a logarithmic equivalent of $\phi_c(s,x,y)$. Precisely, for the pulsating front $\phi_c(x\cdot e-ct,x,y)$ with $c>c_0(e)$, if there is $c'<c$ such that there exists a pulsating front $\phi_{c'}(x\cdot e-c't,x,y)$, then 
\be\label{logeq}
\ln \phi_c(s,x,y)\sim \lambda_c s \hbox{ as $s\rightarrow +\infty$ uniformly in $(x,y)\in \overline{\Omega}$},
\ee
where $\lambda_c$ is the minimal root of $k(\lambda)+c\lambda=0$. If there is a pulsating front $\phi_{c_0}(x\cdot e-c_0(e) t,x,y)$, then
$$\ln \phi_{c_0}(s,x,y)\sim \lambda_{c_0} s \hbox{ as $s\rightarrow +\infty$ uniformly in $(x,y)\in \overline{\Omega}$},$$
where $\lambda_{c_0}$ is the unique root of $k(\lambda)+c_0\lambda=0$.

In this paper, we consider the case $c^{*}(e)>c_0(e)$. This case may happen for instance if $f(x,y,u)=f(u)$ with $f_u(0)>0$ being small enough and $\int_{0}^1 f(u)du$ being large enough, see the ZFK model \cite{Ve} arising in combustion theory. We call fronts with speed $c^*(e)>c_0(e)$ pushed fronts since these fronts are pushed by its main part, instead of being pulled by their exponential tail. The pulled and pushed terminology comes from the stability studies in the monostable case, see \cite{Rothe,Sto,van}. We refer to \cite{GGHR} for investigation of  the inside structure of traveling fronts through the pulled and pushed terminology. In the sequel, we drop $e$ in $c^*(e)$ and $c_0(e)$ for convenience if there is no confusion.
From \cite{H2008}, one knows that for each $c> c_0$, the set
$$F_c=\{\lambda\in (0,+\infty); k(\lambda)+c\lambda=0\},$$
is either a singleton or of two points and for $c=c_0$, the set $F_c$ is either empty or a singleton. Roughly speaking, this is because $k(\lambda)$ is concave and satisfies $k(0)=\mu_0$.

 The following theorem gives the exact exponential behavior of the pushed front $\phi(x\cdot e -c^* t,x,y)$ with the minimal speed $c^{*}>c_0$ as $x\cdot e-c^* t\rightarrow +\infty$.

\begin{theorem}\label{Th1}
Assume that there is a pulsating front $\phi_{c^*}(x\cdot e- c^*t,x,y)$ of \eqref{RD} with $c^*>c_0$. Then there exist $0<\lambda_*<\lambda_{*}^+<+\infty$ such that they are roots of $k(\lambda)+c^{*}\lambda=0$ and there exists $B>0$ such that
\be\label{exprate}
\phi_{c^*}(s,x,y)\sim B e^{-\lambda_*^+ s}\varphi_{*}^+(x,y)\hbox{ as } s\rightarrow +\infty,
\ee
uniformly in $(x,y)\in\overline{\Omega}$, where $\varphi_{*}^+(x,y)$ is the principal eigenfunction of $L_{\lambda_*^+} \varphi=k(\lambda_*^+) \varphi$ with $\|\varphi_*^+\|_{L^{\infty}(\mathcal{C})}=1$.
\end{theorem}

\begin{remark}
From this result and Theorem~1.3 of \cite{H2008}, one can get a fact that if there is a pulsating front $\phi_{c^*}(x\cdot e- c^*t,x,y)$, the KPP assumption~\eqref{KPP} means $c^*=c_0$, where $c^*$ is defined by \eqref{c-star}. This implies that if we assume $c^*>c_0$, the KPP assumption \eqref{KPP} can not hold.
\end{remark}

 By \eqref{logeq}, one knows a logarithmic equivalent of the exponential behavior approaching to $0$ of $\phi(x\cdot e-ct,x,y)$ with $c>c^*$ as $x\cdot e -ct\rightarrow +\infty$. However, from the proof of Theorem~\ref{Th1}, we can make the logarithmic equivalent more precise to be  the exponential behavior.

\begin{theorem}\label{Th2}
Assume that  $\phi(x\cdot e- ct,x,y)$ is a pulsating front of \eqref{RD} with $c>c^*>c_0$. Then there is $B>0$ such that 
$$
\phi_c(s,x,y)\sim B e^{-\lambda_c s}\varphi_{\lambda_c}(x,y)\hbox{ as } s\rightarrow +\infty,
$$
uniformly in $(x,y)\in\overline{\Omega}$, where $\lambda_c=\min F_c$ and $\varphi_{\lambda_c}(x,y)$ 
is the principal eigenfunction of $L_{\lambda_c} \varphi=k(\lambda_c) \varphi$ with $\|\varphi_{\lambda_c}\|_{L^{\infty}(\mathcal{C})}=1$.
\end{theorem}

To compare the results in \cite{H2008} and above results, we refer to Tables~1 and 2. One can see that there is still one case remained unclear, that is, the exponential behavior of $\phi_c$ when $c=c_0$ with general assumption. In this case, the exponential behavior of $\phi_c$ could be either 
$B s^{2m+1} e^{-\lambda_{c_0} s}\psi_{\lambda_{c_0}}$ or $B e^{-\lambda_{c_0} s}\psi_{\lambda_{c_0}}$.

\begin{table}[ht]
\centering
\begin{tabular}{|c|c|c|}
\hline 
\multicolumn{2}{|c|}{With KPP assumption ($c^*=c_0$)} \\ \cline{1-2}
If $c>c^*=c_0$ & $\phi_c(s,x,y)\sim B e^{-\lambda_c s} \psi_{\lambda_c}$ as $s\rightarrow +\infty$\\ 
\hline
If $c=c^*=c_0$ & $\phi_c(s,x,y)\sim B s^{2m+1} e^{-\lambda_{c_0} s}\psi_{\lambda_{c_0}}$ as $s\rightarrow +\infty$\\
\hline
\multicolumn{2}{|c|}{Without KPP assumption} \\ \cline{1-2}
If $c>c^*\ge c_0$ & $\ln \phi_c(s,x,y)\sim \lambda_c s$ as $s\rightarrow +\infty$\\ 
\hline
If $c=c^*=c_0$ & $\ln \phi_{c}(s,x,y)\sim \lambda_{c_0} s$ as $s\rightarrow +\infty$\\
\hline
\end{tabular}
\caption{Results in \cite{H2008}.}
\end{table}

\begin{table}[ht]
\centering
\begin{tabular}{|c|c|c|}
\hline 
\multicolumn{2}{|c|}{General assumption} \\ \cline{1-2}
If $c>c^*\ge c_0$ & $\phi_c(s,x,y)\sim B e^{-\lambda_c s} \psi_{\lambda_c}$ as $s\rightarrow +\infty$\\ 
\hline
If $c=c^*>c_0$ & $\phi_c(s,x,y)\sim B e^{-\lambda_*^+ s}\varphi_{*}^+$ as $s\rightarrow +\infty$\\
\hline
\end{tabular}
\caption{Results in Theorems~\ref{Th1} and \ref{Th2}.}
\end{table}

Indeed, as soon as we have the exponential behavior of the pulsating front $\phi_c$, we can immediately have the uniqueness from Theorem~2.2 of \cite{HR2011}.

\begin{theorem}
If $U_1(t,x,y)=\phi_1(x\cdot e -ct,x,y)$ and $U_2(t,x,y)=\phi_1(x\cdot e -ct,x,y)$ are two pulsating fronts with either $c=c^*>c_0$ or $c>c^*>c_0$, then there exists $\sigma\in \R$ such that
$$\phi_1(s,x,y)=\phi_2(s+\sigma,x,y) \hbox{ for all $(s,x,y)\in \R\times\overline{\Omega}$},$$
that is, there exists $\tau\in \R$ $(\tau=-\sigma/c)$ such that
$$U_1(t,x,y)=U_2(t+\tau,x,y) \hbox{ for all $(t,x,y)\in \R\times\overline{\Omega}$}.$$
\end{theorem}

Now, we concern the global stability of the pushed front. Here, we only consider that $x$ is in one dimension, that is, $d=1$. The following result shows that the solution of the Cauchy problem \eqref{RD} converges to a pushed front once the initial value $u_0(x,y)$ is front-like and decays as fast as $e^{-r x}$ as $x\rightarrow +\infty$, where $r>\lambda_*$ and $\lambda_*$ is the minimal root of $k(\lambda) +c^* \lambda=0$. For $x$ being in high dimensional spaces, the solution of the Cauchy problem with the front-like initial value $u_0(x,y)$ decaying as $e^{-r x\cdot e}$ as $x\cdot e\rightarrow +\infty$ can not always be expected to converge to a shift of $\phi_{c^*}(x\cdot e-c^* t,x,y)$ in general and oscilating solutions between two shifts of $\phi_{c^*}(x\cdot e-c^* t,x,y)$ may appear as in the bistable reaction case, see \cite{MNT} for bistable reactions.

\begin{theorem}\label{Th3}
For $d=1$, assume that there is a pulsating front $U(t,x,y)=\phi(x- c^*t,x,y)$ with $c^{*}>c_0$. There exists $\varepsilon_0>0$ such that if $u_0(x,y)\in [0,1]$ satisfies
$$\limsup_{\xi\rightarrow -\infty} \sup_{(x,y)\in\overline{\Omega},\ x\le \xi}(1-u_0(x,y))\le \varepsilon_0,$$
and for $r>\lambda_{*}$
$$\limsup_{\xi\rightarrow +\infty} \sup_{(x,y)\in\overline{\Omega},\ x\ge \xi} e^{r x}u_0(x,y)<+\infty,$$
then there exist $\tau\in \R$ such that
$$u(t,x,y)\rightarrow U(t+\tau,x,y) \hbox{ as $t\rightarrow +\infty$ uniformly in $\overline{\Omega}$}.$$
\end{theorem}

Notice that the global stability of a pulled front requires the initial value $u_0(x,y)$ decays as the same rate as the pulled front itself and the convergence does not have a shift in time, see \cite{HR2011}. While, for a pushed front, the decaying rate of the initial value $u_0(x,y)$ could be wider and a shift in time in the convergence occurs in general. 

We organize this paper as following. In Section~2, we prove Theorem~\ref{Th1} and Theorem~\ref{Th2}. Section~3 is devoted to the proof of stability of $\phi(x-c^*t,x,y)$.


\section{Exponential decaying rate of $\phi(s,x,y)$ as $s\rightarrow +\infty$}

This section is devoted to the proof of Theorems~\ref{Th1} and \ref{Th2}. 

\subsection{Preliminaries}
We first recall a lemma of \cite{H2008} for reader's convenience. 

\begin{lemma}\label{lemFc}
The function $k(\lambda)$ is analytic and concave in $\R$. For each $c>c_0$, the set 
$$F_c=\{\lambda\in (0,+\infty); k(\lambda)+c\lambda=0\}$$
is either a singleton or of two points.
\end{lemma}

Then, we prove a comparison principle.

\begin{lemma}\label{comparison}
Let $\overline{U}$ and $\underline{U}$ be respectively classical
supersolution and subsolution of
\begin{eqnarray*}
\left\{\begin{array}{lll}
\overline{U}_t-\nabla \cdot (A(z)\nabla \overline{U})+q(z)\cdot \nabla \overline{U} \ge f(z,\overline{U}), &&t\in\R, z=(x,y)\in \overline{\Omega},\\
\nu A\nabla \overline{U}\ge 0, &&t\in\R, z\in\partial\Omega,
\end{array}
\right.
\end{eqnarray*}
and
\begin{eqnarray*}
\left\{\begin{array}{lll}
\underline{U}_t-\nabla \cdot (A(z)\nabla \underline{U})+q(z)\cdot \nabla \underline{U} \le f(z,\underline{U}), &&t\in\R, z=(x,y)\in \overline{\Omega},\\
\nu A\nabla \underline{U}\le 0, &&t\in\R, z\in\partial\Omega.
\end{array}
\right.
\end{eqnarray*}
Assume that $\overline{U}(t,x,y)=\overline{\phi}(x\cdot e -ct,x,y)$ and $\underline{U}(t,x,y)=\underline{\phi}(x\cdot e -ct,x,y)$ where $\overline{\phi}$ and $\underline{\phi}$ are periodic in $(x,y)$. If there are $\sigma_1<\sigma_2$ such that
\begin{eqnarray}\label{opup}
\left\{\begin{array}{lll}
\overline{\phi}(\sigma_1,x,y)\ge \underline{\phi}(\sigma_1,x,y), &&\hbox{ for all $(x,y)\in\overline{\Omega}$},\\
\overline{\phi}(\sigma_2,x,y)\ge \underline{\phi}(\sigma_2,x,y),  &&\hbox{ for all $(x,y)\in\overline{\Omega}$},\\
\overline{\phi}(s,x,y)>\underline{\phi}(\sigma_2,x,y),  &&\hbox{ for all $(x,y)\in\overline{\Omega}$ and $\sigma_1< s<\sigma_2$},\\
\overline{\phi}(\sigma_1,x,y)>\underline{\phi}(s,x,y),  &&\hbox{ for all $(x,y)\in\overline{\Omega}$ and $\sigma_1< s< \sigma_2$},
\end{array}
\right.
\end{eqnarray}
then
$$\overline{\phi}(s,x,y)\ge \underline{\phi}(s,x,y), \hbox{ for $\sigma_1\le s\le \sigma_2$ and $(x,y)\in \overline{\Omega}$},$$
that is $\overline{U}(t,x,y)\ge \underline{U}(t,x,y)$ for all $(t,x,y)\in\R\times\overline{\Omega}$ such that $\sigma_1\le x\cdot e-ct\le \sigma_2$.
\end{lemma}

\begin{proof}
Set 
$$\tau_*=\inf\{\tau\in\R; \overline{\phi}(s,x,y)\ge \underline{\phi}(s+\tau,x,y) \hbox{ for $\sigma_1\le s\le \sigma_2-\tau$ and $(x,y)\in\overline{\Omega}$}\}.$$
Notice that $0\le \tau_*\le \sigma_2-\sigma_1$ and one only needs to prove $\tau_*=0$. Assume that $\tau_*>0$. Then, since $\overline{\phi}$ and $\underline{\phi}$ is periodic in $(x,y)$, there are sequences $0\le \tau_n<\tau_*$ and $(s_n,x_n,y_n)\in [\sigma_1,\sigma_2-\tau_n]\times\mathcal{C}$ such that 
\be\label{snxnyn}
\tau_n\rightarrow \tau_* \hbox{ as $n\rightarrow +\infty$ and } \overline{\phi}(s_n,x_n,y_n)\le \underline{\phi}(s_n+\tau_n,x_n,y_n).
\ee
Then, by \eqref{opup}, there are $s_0\in (\sigma_1,\sigma_2-\tau_*)$ and $(x_0,y_0)\in \mathcal{C}$ such that $s_n\rightarrow s_0$ and $(x_n,y_n)\rightarrow (x_0,y_0)$ as $n\rightarrow +\infty$. By passing $n$ to $+\infty$ in \eqref{snxnyn} and definition of $\tau_*$, one has that
$$\overline{\phi}(s_0,x_0,y_0)=\underline{\phi}(s_0+\tau_*,x_0,y_0).$$

If $c=0$, then $\overline{U}(t,x,y)=\overline{U}(x,y)=\overline{\phi}(x\cdot e,x,y)$, $\underline{U}'(t,x,y)=\overline{U}'(x,y):=\overline{\phi}(x\cdot e+\tau_*,x,y)$ and $\overline{U}(x_0,y_0)=\underline{U}'(x_0,y_0)$. By applying the maximum principle for the elliptic equation, it follows that $\overline{\phi}(x\cdot e,x_0,y_0)\equiv \underline{\phi}(x\cdot e+\tau_*,x,y)$ which contradicts \eqref{opup}.
If $c\neq 0$, let $t_0=(x_0\cdot e-s_0)/c$ and $\underline{U}'(t,x,y):=\overline{\phi}(x\cdot e-ct +\tau_*,x,y)$. Then, $\overline{U}(t_0,x_0,y_0)=\underline{U}'(t_0,x_0,y_0)$. By applying the maximum principle for the parabolic equation, it follows that $\overline{\phi}(x\cdot e -ct,x_0,y_0)\equiv \underline{\phi}(x\cdot e -ct +\tau_*,x,y)$ which also contradicts \eqref{opup}.

Thus, $\tau_*=0$ and it completes the proof.
\end{proof}

\subsection{Proof of Theorem~\ref{Th1}}
This subsection is devoted to the proof of Theorem~\ref{Th1}. In this subsection, we always assume that there is a pulsating front $\phi_{c^*}(x\cdot e-c^* t,x,y)$ of \eqref{RD} with $c^*>c_0$ and we simply denote $\phi_{c^*}$ by $\phi$. Since $c^*>c_0$, it follows from Lemma~\ref{lemFc} that the set $F_{c^*}$ is either a singleton or of two points. Define 
$$\lambda_c=\min\{\lambda>0; k(\lambda)+c\lambda=0\},$$
and $\lambda_*:=\lambda_{c^*}$.
The basic idea of proving \eqref{exprate} is to exclude the case of the  decaying rate $e^{-\lambda_* s}$. This can further deduce that $F_{c^*}$ is of two points, that is, there is a point $\lambda_*^+>\lambda_*$ in $F_{c^*}$ and the decaying rate can only be $e^{-\lambda_*^+ s}$ as $s\rightarrow +\infty$. This idea is inspired by \cite{R} for the reaction-diffusion-advection equation in cylinders. But since our case is spatially periodic, things become more complicated. 

Assume by contradiction that
\be\label{asy1}
0<\liminf_{s\rightarrow +\infty} \Big[\min_{(x,y)\in \overline{\Omega}}\Big(\frac{\phi(s,x,y)}{e^{-\lambda_* s}}\Big)\Big]\le \limsup_{s\rightarrow +\infty}\Big[\max_{(x,y)\in \overline{\Omega}}\Big(\frac{\phi(s,x,y)}{e^{-\lambda_* s}}\Big)\Big]<+\infty.
\ee
Notice that Proposition~4.3 of \cite{H2008} does not need the KPP assumption~\eqref{KPP}. It implies that once \eqref{asy1} is satisfies, one can follow the proof of Theorem~1.3 of \cite{H2008} that there exists $B>0$ such that 
\be\label{asy-phi}
\phi(s,x,y)\sim B e^{-\lambda_* s}\varphi_*(x,y), \hbox{ as $s\rightarrow +\infty$ uniformly in $(x,y)\in\overline{\Omega}$},
\ee
where $\varphi_*(x,y)$ is the unique positive principal eigenfunction of $L_{\lambda_*}\varphi=k(\lambda_*) \varphi$ such that $\|\varphi_{\lambda_*}\|_{L^{\infty}(\mathcal{C})}=1$.  One can further deduce that

\begin{lemma}\label{lemma2.1}
There exist positive constants $s_0$ and $\theta$ such that
$$\phi(s,x,y)\le B e^{-\lambda_* s}\varphi_*(x,y)+O(e^{-r s}), \hbox{ for $s\ge s_0$ and $(x,y)\in \overline{\Omega}$},$$
$$\phi_s(s,x,y)\sim -B\lambda_* e^{-\lambda_* s}\varphi_*(x,y) \hbox{ as $s\rightarrow +\infty$},$$
and
$$|\phi_s(s,x,y)|\le B\lambda_* e^{-\lambda_* s}\varphi_*(x,y)+O(e^{-r s}),  \hbox{ for $s\ge s_0$ and $(x,y)\in \overline{\Omega}$},$$
where $r=\min\{(1+\alpha)\lambda_*,(\lambda_*+\lambda_*^+)/2\}$ and $\alpha$ is defined by that $f(x,y,u)$ is of $C^{1,\alpha}$ with respect to $u$.
\end{lemma}

\begin{proof}
Let $\rho(s,x,y):=B e^{-\lambda_* s}\varphi_*(x,y)$ and $V(s,x,y):=\phi(s,x,y)-\rho(s,x,y)$. By \eqref{asy-phi}, one has that
$$\liminf_{s\rightarrow +\infty} \Big[\min_{(x,y)\in \overline{\Omega}}\Big(\frac{V(s,x,y)}{e^{-\lambda_* s}}\Big)\Big]=0.$$
Then, for any sequence $\{\varepsilon_n\}_{n\in N}$ such that $0<\varepsilon_n<1$ and $\varepsilon_n\rightarrow 0$ as $n\rightarrow +\infty$, there is a sequence $\{s_n\}_{n\in N}$ such that $s_0<s_1<\cdots<s_n<\cdots$ and 
\be\label{sn}
s_n\rightarrow +\infty \hbox{ as $n\rightarrow +\infty$ and } V(s_n,x,y)\le \varepsilon_n e^{-\lambda_* s_n}\varphi_*(x,y) \hbox{ for all $n\in N$}.
\ee

Since $\phi(x\cdot e -c^* t,x,y)$ satisfies \eqref{RD} and $\rho(x\cdot e-c^* t,x,y)$ satisfies the linearised equation of \eqref{RD} at $0$, one has that $v(t,x,y):=V(x\cdot e -c^* t,x,y)$ satisfies 
\begin{eqnarray}\label{eq-v}
\left\{\begin{array}{lll}
v_t-\nabla \cdot (A(z)\nabla v)+q(z)\cdot \nabla v =f(z,v+\rho)-f_u(z,0)\rho, &&t\in\R, z=(x,y)\in \overline{\Omega},\\
\nu A\nabla v=0, &&t\in\R, z\in\partial\Omega,
\end{array}
\right.
\end{eqnarray}
For $t\in\R$ and $(x,y)\in\overline{\Omega}$, define 
$$\overline{v}_n(t,x,y)=\overline{\phi}_n(x\cdot e-ct,x,y)=\theta e^{-r (x\cdot e -c^* t)}\varphi_r(x,y)+ \varepsilon_n e^{-\lambda_* (x\cdot e -c^* t)}\varphi_*(x,y).$$
where $r=\min\{(1+\alpha)\lambda_*,(\lambda_*+\lambda_*^+)/2\}$ and $\theta$ is a constant to be given. Notice that $k(r)+c^* r>0$ since $k$ is concave by Lemma~\ref{lemFc}. Since $f_u(x,y,u)$ is of $C^{0,\alpha}$ with respect to $u$, there are $\gamma\in (0,1)$ and $\delta>0$ such that
\be\label{eq-F}
|f(x,y,u)-f_u(x,y,0)u|\le \delta u^{1+\alpha} \hbox{ for all $(x,y,u)\in\overline{\Omega}\times [0,\gamma]$}.
\ee
Let
$$ \theta:=e^{(r-\lambda_*)s_0} \frac{\displaystyle\min_{\overline{\Omega}}\varphi_{*}}{\displaystyle\max_{\overline{\Omega}}\varphi_{r}}.$$
Notice that $\theta$ is increasing in $s_0$.
Since $V(x\cdot e -c^*t,x,y)\rightarrow 0$ as $x\cdot e-c^* t\rightarrow +\infty$, one can take a sufficiently large constant $\tau$ such that
$$\theta e^{-r s_0}\varphi_r(x,y)>V(x\cdot e -c^*t+\tau,x,y), \hbox{ for all $(t,x,y)\in\R\times\overline{\Omega}$ such that $x\cdot e -c^* t\ge s_0$}.$$
It means that
\be\label{phiV}
\overline{\phi}_n(s_0,x,y)> V(x\cdot e -c^*t+\tau,x,y), 
\ee
 for all $n\in\mathbb{N}$ and $(t,x,y)\in\R\times\overline{\Omega}$ such that $x\cdot e -c^* t\ge s_0$.
By \eqref{sn}, one also knows that
\be\label{phiV-}
\overline{\phi}_n(x\cdot e -c^*t,x,y)> V(s_n+\tau,x,y), \hbox{ for $(t,x,y)\in\R\times\overline{\Omega}$ such that $s_0\le x\cdot e -c^* t \le s_n$}.
\ee
By the definition of $\theta$, one can take $s_0$ sufficiently large such that 
\be\label{c-theta}
\theta e^{-r s} \varphi_r(x,y)\le e^{-\lambda_* s}\varphi_*(x,y), \hbox{ for $s\ge s_0$}
\ee
and
\be\label{c-theta2}
(k(r)+c r) \theta\ge \delta [(2+B)\max_{(x,y)\in\overline{\Omega}}\varphi_*(x,y)]^{1+\alpha}, 
\ee
since $k(r)+c^* r>0$.
Even if it means increasing $s_0$, one can also assume without loss of generality that
$$\overline{v}_n(t,x,y)+\rho(x\cdot e-c^* t,x,y)=\theta e^{-r (x\cdot e -c^* t)}\varphi_r(x,y)+ (B+\varepsilon_n) e^{-\lambda_* (x\cdot e -c^* t)}\varphi_*(x,y)\le \gamma,$$
for all $n\in\mathbb{N}$ and $(t,x,y)\in\R\times\overline{\Omega}$ such that $x\cdot e -c^* t\ge s_0$.

Now, we can check from \eqref{eq-F}-\eqref{c-theta2} that 
\begin{align*}
\mathscr{L}\overline{v}_n:=& (\overline{v}_n)_t-\nabla \cdot (A(z)\nabla \overline{v}_n)+q(z)\cdot \nabla \overline{v}_n -f(z,\overline{v}_n+\rho)+f_u(x,y,0)\rho\\
=& (k(r)+c^* r)\theta e^{-r (x\cdot e -c^*t)}\varphi_r(x,y) +f_u(x,y,0)\overline{v}_n-f(x,y,\overline{v}_n+\rho)+f_u(z,0)\rho\\
\ge& (k(r)+c^* r)\theta e^{-r (x\cdot e -c^*t)}\varphi_r(x,y) -\delta (\overline{v}_n+\rho)^{1+\alpha}\\
\ge& (k(r)+c^* r)\theta e^{-r (x\cdot e -c^* t)}\varphi_r(x,y) - \delta [(2+B)\max_{(x,y)\in\overline{\Omega}}\varphi_*(x,y)]^{1+\alpha} e^{-(1+\alpha)\lambda_* (x\cdot e -c^* t)}\\
\ge& 0,
\end{align*}
for $(t,x,y)\in\R\times\overline{\Omega}$ such that $x\cdot e -c^* t\ge s_0$ and $\nu A\nabla\overline{v}_n=0$ on $\R\times\partial\Omega$.
By \eqref{phiV}, \eqref{phiV-} and Lemma~\ref{comparison}, it follows that $\overline{v}_n(t,x,y)$ is a supersolution of \eqref{eq-v} and 
$$V(x\cdot e -c^*t+\tau,x,y)\le \overline{\phi}_n(x\cdot e-ct,x,y),$$ 
for $(t,x,y)\in \R\times\overline{\Omega}$ such that $s_0\le x\cdot e -c^* t \le s_n$ and all $n\in \mathbb{N}$.
By passing $n$ to $+\infty$, one has that
$$V(x\cdot e -c^*t+\tau,x,y)\le \theta e^{-r (x\cdot e -c^* t)}\varphi_r(x,y), \hbox{ for $(t,x,y)\in \R\times\overline{\Omega}$ such that $x\cdot e -c^* t \ge s_0$}.$$

By standard parabolic estimates applied to $v(t,x,y)=V(x\cdot e-c^* t,x,y)$, there exists $C_1>0$ such that
$$|V_s(s,x,y)|\le C_1 e^{-r s} \hbox{ for $s\ge s_0$ and $(x,y)\in \overline{\Omega}$}.$$
Notice that $\phi_s(s,x,y)=\rho_s(s,x,y)+V_s(s,x,y)$. Thus,
$$\phi_s(s,x,y)\sim -B\lambda_* e^{-\lambda_* s}\varphi_*(x,y) \hbox{ as $s\rightarrow +\infty$},$$
and there is $s_0\in\R$ such that
$$|\phi_s(s,x,y)|\le B\lambda_* e^{-\lambda_* s}\varphi_*(x,y)+O(e^{-r s})  \hbox{ for $s\ge s_0$ and $(x,y)\in \overline{\Omega}$}.$$

This completes the proof.
\end{proof}
\vskip 0.3cm 

Now, we want to deduce a contradiction. If we can find a pulsating front with speed less than $c^*$, then it contradicts the definition of $c^*$. Notice that $\phi(s,x,y)$ satisfies 
\begin{eqnarray}
\left\{\begin{array}{lll}
\mathcal{L}_{c^*} \phi - f(x,y,\phi)=0, &&\hbox{ in $\R\times\overline{\Omega}$},\\
\nu A e\phi_s +\nu A\nabla\phi=0, &&\hbox{ on $\R\times\partial\Omega$}.
\end{array}
\right.
\end{eqnarray}
where 
$$\mathcal{L}_{c^*} v=-(e\partial_s +\nabla)\cdot (A(e\partial_s+\nabla) v) +q\cdot (\nabla v +e v_s) - c^* v_s.$$ 
For $\tau\in\R$, we now look for  a solution $\Phi(s,x,y)$ of 
\begin{eqnarray}
\left\{\begin{array}{lll}
\mathcal{L}_{c^*-\tau} \Phi - f(x,y,\Phi)=0, &&\hbox{ in $\R\times\overline{\Omega}$},\\
\nu A e\Phi_s +\nu A\nabla\Phi=0, &&\hbox{ on $\R\times\partial\Omega$}.
\end{array}
\right.
\end{eqnarray}
with the form
$$\Phi(s,x,y)=\phi(s,x,y)+(B e^{-\lambda_{c^* -\tau} s} \varphi_{c^*-\tau}(x,y)-B e^{-\lambda_* s} \varphi_*(x,y) )h(s+\theta(x,y)) +\omega(s,x,y),$$
where $\varphi_{c^*-\tau}(x,y)$  is the unique positive principal eigenfunction of $L_{\lambda_{c^*-\tau}}\varphi=k(\lambda_{c^*-\tau}) \varphi$ such that $\|\varphi_{c^*-\tau}\|_{L^{\infty}(\mathcal{C})}=1$, $h(s)$ is a smooth function such that $h(s)=0$ for $s\le -1$ and $h(s)=1$ for $s\ge 0$ and $\theta(x,y)$ is a $C^2(\overline{\Omega})$ nonpositive periodic function such that
\be\label{F-theta}
\nu A\nabla \theta +\nu A e=0 \hbox{ on $\partial\Omega$}.
\ee
An example of the function $\theta(x,y)$ is given in p.364 of \cite{HR2011}.
Then, it implies that $\omega(s,x,y)$ satisfies
\begin{eqnarray}\label{eq-omega}
\left\{\begin{array}{ll}
\mathcal{L}_{c^*-\tau}\Big((B e^{-\lambda_{c^* -\tau} s} \varphi_{c^*-\tau} -B e^{-\lambda_* s} \varphi_* )h(s+\theta)+\omega\Big) &\\
\hskip 4.5cm -f(x,y,\Phi)+f(x,y,\phi)+ \tau \phi_s =0,  &\hbox{ in $\R\times\overline{\Omega}$},\\
\nu A e\omega_s +\nu A\nabla\omega=0,  &\hbox{ on $\R\times\partial\Omega$}.
\end{array}
\right.
\end{eqnarray}
Let $\rho(s)=1+e^{2 r s}$ where $r=\min\{(1+\alpha)\lambda_*,(\lambda_*+\lambda_*^+)/2\}$. Define the weighted $L^2$ spaces
$$L_{\rho}^2=\left\{\omega(s,x,y),\ (s,x,y)\in \R\times\overline{\Omega}; \|\omega\|_{2,\rho}=\left(\int_{\R\times \mathcal{C}} \omega^2 \rho ds dx dy\right)^{\frac{1}{2}} <+\infty\right\},$$
$$H_{\rho}^1=\left\{\omega(s,x,y)\in L_{\rho}^2;\|\omega\|^2_{H^1_{\rho}}=\|\omega\|^2_{L^2_{\rho}}+ \int_{\R\times \mathcal{C}} |\nabla \omega|^2 \rho ds dx dy <+\infty\right\}.$$
Define an operator
$$\mathcal{L} \omega:=\mathcal{L}_{c^*}\omega -f_u(x,y,\phi)\omega,$$
where
$$\omega\in D(\mathcal{L})=\{\omega; \omega\in H^1_{\rho}, (e\partial_s +\nabla)^2 \omega \in L^2_{\rho},\ \nu A e\omega_s +\nu A\nabla\omega=0 \hbox{ on $\R\times\partial\Omega$}\}.$$
We need some properties of the operator $\mathcal{L}$.

\begin{lemma}\label{lemma-M}
For $\beta>0$ sufficiently large, the operator $\mathcal{L}\omega +\beta \omega: D(\mathcal{L})\rightarrow L^2_{\rho}$ is invertible.
\end{lemma}

\begin{proof}
Take any constant $\beta>0$ such that 
$$\beta>\sup_{(x,y)\in \overline{\Omega}, u\in [0,1]} |f_u(x,y,u)|.$$ 
Define $M\omega:=\mathcal{L}\omega +\beta \omega$. Let $\omega$ be in the kernel of $M$. Multiply $M\omega=0$ by $\omega$ and integrate by parts. One gets that
$$\int_{\R\times\mathcal{C}} (e\partial_s \omega +\nabla\omega) A (e\partial_s \omega +\nabla\omega) +(\beta -f_u(x,y,\phi))\omega^2=0,$$
which implies that $\omega\equiv 0$. In $L_{\rho}^2$, the adjoint operator $M^*$ of $M$ is defined by 
$$M^* v=-\frac{1}{\rho}(e\partial_s +\nabla)\cdot (A(e\partial_s+\nabla) (\rho v)) -\frac{1}{\rho}q\cdot (\nabla  +e\partial_s)(\rho v) + \frac{1}{\rho} c^* (\rho v)_s +(\beta-f_u(x,y,\phi))v,$$
with the boundary condition $\nu A e(\rho v)_s +\nu A\nabla(\rho v)=0$ on $\R\times\partial\Omega$.
Let $v$ be in the kernel of $M^*$ and $\omega=\rho v$. Then, $\omega$ satisfies 
\begin{eqnarray*}
\left\{\begin{array}{lll}
-(e\partial_s +\nabla)\cdot (A(e\partial_s+\nabla) \omega) -q\cdot (\nabla \omega +e\omega_s)+ c^* \omega_s +(\beta-f_u(x,y,\phi))\omega=0, &\hbox{ in $\R\times\overline{\Omega}$},\\
\nu A e\omega_s +\nu A\nabla\omega=0, &\hbox{ on $\partial\Omega$}.
\end{array}
\right.
\end{eqnarray*}
The same arguments yield that $\omega\equiv 0$. The fact that the range of $M$ is closed follows from the proof of Lemma~2.5 of \cite{X}. This completes the proof.
\end{proof}

\begin{lemma}
The operator $\mathcal{L}: D(\mathcal{L})\rightarrow L^2_{\rho}$ is invertible.
\end{lemma}

\begin{proof}
{\it Step~1: the kernel of $\mathcal{L}$ is reduced to $\{0\}$.} 
Let $\omega\in D(\mathcal{L})$ such that $\mathcal{L}\omega=0$ and $\omega\not\equiv 0$. By rewriting the equation $\mathcal{L}\omega=0$ in its parabolic form and by parabolic regularity theory, it follows that $\omega$ is of class $C^2(\R\times\Omega)$ and it is a bounded classical solution of $\mathcal{L}\omega=0$ and $\nu A e \omega_s + \nu A\nabla\omega =0$ on $\R\times\partial\Omega$ such that $\omega(\pm \infty,\cdot,\cdot)=0$ and $\omega(s,x,y)=O(e^{-r s})$ as $s\rightarrow +\infty$. Notice that $\phi_s$ is a solution of $\mathcal{L}\omega=0$ and  $\nu A e \omega_s + \nu A\nabla\omega =0$ on $\R\times\partial\Omega$. Take $s_0>0$ such that 
$$|f_u(x,y,1)-f_u(x,y,\phi)|\le \frac{\mu_1}{2} \hbox{ for $s\le -s_0$ and $(x,y)\in \overline{\Omega}$},$$
where $\mu_1$ is defined by \eqref{mu1}.
Since $\phi_s \sim -B \lambda_* e^{-\lambda_* s} \varphi_*(x,y)$ as $s\rightarrow +\infty$ and $\lambda_*<r$, there is $\eta_0>0$ such that for $\eta\ge \eta_0$, 
$$\omega_{\eta}(s,x,y):=\omega(s,x,y) -\eta \phi_s(s,x,y)\ge 0, \hbox{ for all $(s,x,y)\in [-s_0,+\infty)\times\overline{\Omega}$}.$$
Let $\varepsilon_*=\inf\{\varepsilon>0; \omega_{\eta}(s,x,y)+\varepsilon\psi_1(x,y)\ge 0, (s,x,y)\in (-\infty,-s_0]\times\overline{\Omega}\}$ where $\psi_1(x,y)$ is defined by \eqref{mu1} with $\|\psi_1\|_{L^{\infty}(\mathcal{C})}=1$. Since $\omega_{\eta}(s,x,y)\rightarrow 0$ as $s\rightarrow -\infty$, one has that $0\le \varepsilon_*<+\infty$. If $\varepsilon_*>0$, one also has that 
\be\label{varpsilon-s}
\omega_{\eta}(s,x,y)+\varepsilon_* \psi_1(x,y)>0, \hbox{ for $(s,x,y)\in \{-s_0\}\times\overline{\Omega}$ and $s\rightarrow -\infty$, $(x,y)\in\overline{\Omega}$}.
\ee
By the definition of $\varepsilon_*$, it follows that there is $(\bar{s},\bar{x},\bar{y})\in (-\infty,-s_0]\times\overline{\Omega}$ such that $\omega_{\eta}(\bar{s},\bar{x},\bar{y})+\varepsilon_* \psi_1(\bar{x},\bar{y})=0$. Notice that $\omega_{\eta}(s,x,y)+\varepsilon_*\psi_1(x,y)$ satisfies 
\begin{eqnarray*}
\left\{\begin{array}{lll}
\mathcal{L}(\omega_{\eta}+\varepsilon_*\psi_1)=\varepsilon_*\mu_1\psi_1 +(f_u(x,y,1)-f_u(x,y,\phi))\varepsilon_*\psi_1\ge \frac{\varepsilon_* \mu_1}{2} \psi_1>0, &\hbox{ in $\R\times\overline{\Omega}$},\\
\nu A e(\omega_{\eta}+\varepsilon_* \psi_1)_s +\nu A\nabla(\omega_{\eta}+\varepsilon_* \psi_1)=0, &\hbox{ on $\R\times\partial\Omega$}.
\end{array}
\right.
\end{eqnarray*}
By the strong maximum principle and Hopf Lemma applied to the parabolic form of $L\omega$, one has that $\omega_{\eta}+\varepsilon_*\psi_1\equiv 0$ for $(s,x,y)\in (-\infty,-s_0]\times\overline{\Omega}$ which contradicts \eqref{varpsilon-s}. Thus, $\varepsilon_*=0$ and $\omega_{\eta}(s,x,y)\ge 0$ for $\eta\ge \eta_0$ and $(s,x,y)\in \R\times\overline{\Omega}$.
Define $\eta_*=\inf\{\eta\in \R; \omega_{\eta}\ge 0 \hbox{ in } \R\times\overline{\Omega}\}$. One has that $-\infty<\eta_*\le \eta_0$. Since $|\phi_s/\omega|>1$ as $s\rightarrow +\infty$, there exists $(s_*,x_*,y_*)\in [-s_0,+\infty)\times\overline{\Omega}$ such that $$\omega_{\eta_*}(s_*,x_*,y_*)=0.$$
Otherwise, $\omega_{\eta_*}>0$ in $[-s_0,+\infty)\times\overline{\Omega}$ and then $\omega_{\eta_*-\varepsilon}\ge 0$ in $[-s_0,+\infty)\times\overline{\Omega}$ for some small $\varepsilon>0$ by continuity. By above arguments, one can deduce that $\omega_{\eta_*-\varepsilon}\ge 0$ in $\R\times\overline{\Omega}$. This contradicts the definition of $\eta_*$. The existence of $(s_*,x_*,y_*)$ implies that there exists $(t_*,x_*,y_*)$ such that $\omega_{\eta_*}(x_*\cdot e - c t_*,x_*,y_*)=0$. Since $\omega_{\eta_*}$ satisfies $\mathcal{L}(\omega)=0$ in its parabolic form, the strong maximum principle and Hopf Lemma show that $\omega_{\eta_*}\equiv 0$, which is impossible. Therefore, the kernel of $\mathcal{L}$ is reduced to $\{0\}$.

{\it Step~2: the range of $\mathcal{L}$ is closed.} Let $\theta$ be  a $C^2(\overline{\Omega})$ nonpositive periodic function satisfying \eqref{F-theta}.
Let $h$ be a smooth function such that $h(s)=0$ for $s\le -1$ and $h(s)=1$ for $s\ge 0$.
Let $d$ be the function defined for all $(s,x,y)\in \R\times\overline{\Omega}$ by 
$$d(s,x,y)=\varphi_r(x,y) e^{-r s} h(s+\theta(x,y)) +\psi_1(x,y)(1- h(s+\theta(x,y))).$$
For $\mathcal{L}\omega=\widetilde{g}$ where $\omega\in D(\mathcal{L})$ and $\widetilde{g}\in L_{\rho}^2$, let 
$$\omega(s,x,y)=d(s,x,y)v(s,x,y) \hbox{ and } \widetilde{g}(s,x,y)=d(s,x,y)g(s,x,y).$$ 
Then, one can get that 
\begin{align*}
\hat{\mathcal{L}}v:=&-(e\partial_s +\nabla)\cdot (A(e\partial_s+\nabla) v) +q\cdot (\nabla v +e v_s) - c^* v_s -f_u(x,y,\phi) v\\
&+\frac{v}{d} \Big(-(e\partial_s +\nabla)\cdot (A(e\partial_s+\nabla) d) +q\cdot (\nabla d +e d_s) - c^* d_s \Big) \\
&-\frac{2}{d}(e\partial_s d +\nabla d)\cdot (A(e\partial_s +\nabla)v)=g, \hbox{ in $\R\times\overline{\Omega}$},
\end{align*}
and
$$\nu A e v_s +\nu A\nabla v=0, \hbox{ on $\R\times\partial\Omega$}.$$
Let $\hat{\mathcal{L}} v_n=g_n\rightarrow g$ in $L^2$. Assume that $\|v_n\|_{L^2}$ is unbounded. Then, let $\omega_n=\frac{v_n}{\|v_n\|_{L^2}}$. One has $\|\omega_n\|_{L^2}=1$. Now, 
\be\label{hatLomega}
\hat{L}\omega_n=\frac{g_n}{\|v_n\|_{L^2}}\rightarrow 0 \hbox{ as $n\rightarrow +\infty$ and } \nu A e(\omega_n)_s +\nu A\nabla\omega_n=0 \hbox{ on $\R\times\partial\Omega$}.
\ee
Notice that $d \omega_n$ satisfies $\mathcal{L}(d \omega_n)=d g_n/\|v_n\|_{L^2}$ and 
$$M(d \omega_n)=\mathcal{L}(d \omega_n)+\beta d \omega_n.$$
By the proofs of Lemmas 2.2-2.4 of \cite{X}, one can get that $d \omega_n$ is bounded in $H^1_{\rho}$ and hence $\omega_n$ is bounded in $H^1$. Then, a subsequence of $\omega_n$ converges in $H^1$ weakly and in $L^2_{loc}$ strongly to some $\omega_0$ in $H^1$. Moreover, $\hat{\mathcal{L}}\omega_0=0$ in $\R\times\overline{\Omega}$ and $\nu A e(\omega_0)_s +\nu A\nabla\omega_0=0$ on $\R\times\partial\Omega$ by \eqref{hatLomega}. Since the kernel of $\mathcal{L}$ is reduced to $\{0\}$, then $\omega_0\equiv 0$. Take $N>0$  sufficiently large such that
$$|f_u(x,y,0)-f_u(x,y,\phi)|<k(r)+c^*r \hbox{ for $s\ge N$ and $(x,y)\in \overline{\Omega}$},$$
and
$$|f_u(x,y,1)-f_u(x,y,\phi)|<\mu_1 \hbox{ for $s\le -N$ and $(x,y)\in \overline{\Omega}$},$$
By the definition of $d(s,x,y)$, one also has that
\begin{align*}
\hat{L}v=&-(e\partial_s +\nabla)\cdot (A(e\partial_s+\nabla) v) +q\cdot (\nabla v +e v_s) - c^* v_s+( k(r)+c^* r +f_u(x,y,0)\\
&-f_u(x,y,\phi)) v -\frac{2}{e^{-r s} \varphi_r}(e\partial_s +\nabla )(e^{-r s} \varphi_r)\cdot (A(e\partial_s +\nabla)v), \hbox{ for $s\ge N$ and $(x,y)\in \overline{\Omega}$},
\end{align*}
and
\begin{align*}
\hat{L}v=&-(e\partial_s +\nabla)\cdot (A(e\partial_s+\nabla) v) +q\cdot (\nabla v +e v_s) - c^* v_s+( \mu_1 +f_u(x,y,0)-f_u(x,y,\phi)) v \\
&-\frac{2}{\psi_1} \nabla\psi_1\cdot (A(e\partial_s +\nabla)v), \hbox{ for $s\le -N$ and $(x,y)\in \overline{\Omega}$},
\end{align*}
Let $\chi(s)$ be a smooth function such that $\chi(s)=0$ for $s\le N$, $\chi(s)=s-N$ for $N\le s\le N+1$ and $\chi(s)=1$ for $s\ge N+1$. Let $m(x,y)$ be the smooth positive solution of the equation
$$\nabla\cdot (A\nabla m)-\nabla\cdot ((q+2rAe-2A\frac{\nabla\varphi_r}{\varphi_r})m)=0.$$
such that $\min_{\mathcal{C}} m(x,y)=1$. Existence of such a solution follows from the fact that the adjoint problem $\nabla\cdot (A\nabla m) +(q+2rAe -2A\nabla\varphi_r/\varphi_r)\nabla m=0$. By integrating \eqref{hatLomega} against $\chi(s) m(x,y) \omega_n$, one gets that
\begin{align*}
\int_{[N,+\infty)\times \mathcal{C}} m \chi (e\partial_s \omega+\nabla \omega_n)A(e\partial_s \omega+\nabla \omega_n) +   (k(r)+c^* r +f_u(x,y,0)-f_u(x,y,\phi))m \chi \omega_n^2\\
+ \int_{[N,N+1]\times\mathcal{C}} \frac{1}{2}\omega_n^2[c^*-2eA\nabla m -(q\cdot e +2r eAe -2\frac{eA\nabla\varphi_r}{\varphi})m]\\
=\int_{[N,+\infty)\times \mathcal{C}} \frac{g_n m \chi \omega_n}{\|v_n\|_{L^2}}\rightarrow 0.
\end{align*}
Since both first two terms on the left side are nonnegative and $\omega_n\rightarrow 0$ in $L^2_{loc}(\R\times\mathcal{C})$, one has that they converge to $0$ as $n\rightarrow +\infty$. In particular, $\int_{[N,+\infty)\times \mathcal{C}}\omega_n^2\rightarrow 0$ as $n\rightarrow +\infty$. Similar arguments can lead to that $\int_{(-\infty,-N]\times \mathcal{C}}\omega_n^2\rightarrow 0$ as $n\rightarrow +\infty$. Thus, $\omega_n$ converges to $0$ strongly in $L^2(\R\times\mathcal{C})$ as $n\rightarrow +\infty$ which contradicts to $\|\omega_n\|_{L^2}=1$. This means that the sequence $v_n$ is bounded in $L^2$.

Since  $M(d v_n)=L(d v_n)+\beta d v_n$, one can get that $d v_n$ is bounded in $H^1_{\rho}$ and hence $v_n$ is bounded in $H^1$. Therefore, a subsequence converges weakly in $H^1$ to some $v_0$ such that $\hat{\mathcal{L}}v_0 =g$.

{\it Step~3: the kernel of $\mathcal{L}^*$ is reduced to $\{0\}$.} Take $\beta>0$ such that
$$\beta>\sup_{(x,y)\in \overline{\Omega}, u\in [0,1]} |f_u(x,y,u)|.$$ 
Then, by Lemma~\ref{lemma-M}, the operator $M=\mathcal{L}+\beta$ is invertible. If there is $v^*\neq 0$ in the kernel of $\mathcal{L}^*$, then the arguments in p. 220 of \cite{X} imply that there is $v\in D(\mathcal{L})$ such that $\mathcal{L} v=0$ which contradicts that the kernel of $\mathcal{L}$ is reduced to $\{0\}$.
\end{proof}
\vskip 0.3cm

Since $k(\lambda)$ and $\varphi_{\lambda}$ are analytic with respect to $\lambda$ (see \cite{H2008}) and $c^*>c_0$, then it follows from the definition of $\lambda_{c^*}$ that $\lambda_{c^*-\tau}$ and $\varphi_{c^*-\tau}$ are smooth with respect to $\tau$ for $\tau$ small enough. Let $G(\omega,\tau)$ be defined by 
$$G(\omega,\tau):=\mathcal{L}_{c^*-\tau}\Big((B e^{-\lambda_{c^* -\tau} s} \varphi_{c^*-\tau} -B e^{-\lambda_* s} \varphi_* )h(s+\theta)+\omega\Big) -f(x,y,\Phi)+f(x,y,\phi)+ \tau \phi_s.$$
Then, by Lemma~\ref{lemma2.1}, one can check that $G$ is a function from $D(\mathcal{L})\times\R$ to $L^2_{\rho}$. One can also easily verify that $G$ is continuous with respect to $(\omega,\tau)$ and continuously differentiable with respect to $\omega$, $\tau$. Notice that $\partial_{\omega} G(0,0)=\mathcal{L}:D(\mathcal{L})\rightarrow L^2_{\rho}$ is invertible.
By the implicit function theory, one has the following lemma.

\begin{lemma}\label{lemma-omega}
For $\tau$ small enough, there is a solution $\omega_{\tau}(s,x,y)$ of \eqref{eq-omega} such that 
$$\omega_{\tau}(s,x,y)= O(e^{-r s }) \hbox{ as $s\rightarrow +\infty$}.$$ 
\end{lemma}

Notice that Lemma~\ref{lemma-omega} contradicts the definition of $c_*$.
Therefore, \eqref{asy1} is violated and $\phi(s,x,y)$ does not decay as the rate $e^{-\lambda_* s}$. Then, we are ready to prove Theorem~\ref{Th1}.
\vskip 0.3cm

\begin{proof}[Proof of Theorem~\ref{Th1}]
{\it Step~1: the upper bound.} Notice that the proof of Proposition~4.3 of \cite{H2008} does not require the KPP assumption. It means that
$$ \limsup_{s\rightarrow +\infty}\Big[\max_{(x,y)\in \overline{\Omega}}\Big(\frac{\phi(s,x,y)}{e^{-\lambda_* s}}\Big)\Big]<+\infty.$$
Then, the fact that $\phi(s,x,y)$ does not decay as the rate $e^{-\lambda_* s}$ implies that
$$\liminf_{s\rightarrow +\infty} \Big[\min_{(x,y)\in \overline{\Omega}}\Big(\frac{\phi(s,x,y)}{e^{-\lambda_* s}}\Big)\Big]=0.$$
Since $\nabla \phi/\phi$ is bounded in $\R\times\overline{\Omega}$ (by the Harnack inequality applied to $\phi(x\cdot e-c^* t,x,y)$) and $\phi$ is periodic in $(x,y)$,  for any $\varepsilon>0$, there is $s_\varepsilon$ such that 
\be\label{s-epsilon}
\phi(s_\varepsilon,x,y)\le  \varepsilon e^{-\lambda_* s_\varepsilon}\varphi_*(x,y).
\ee

Now, let $\lambda_*^+:=\max\{\lambda>0; k(\lambda)+c^*\lambda=0\}$ if $F_{c^*}$ is of two points and let $\lambda_*^+$ be a sufficiently large constant larger than $\lambda_*$ if $F_{c^*}$ is a singleton.
Let $\varphi_*^+$ be the unique positive principal eigenfunction of $L_{\lambda_*^+} \varphi=k(\lambda_*^+) \varphi$ such that $\|\varphi_*^+\|_{L^{\infty}(\mathcal{C})}=1$.
Denote 
$$\kappa=\Big(\min_{(x,y)\in\overline{\Omega}} \psi_*,\min_{(x,y)\in\overline{\Omega}} \psi_*^+\Big) \hbox{ and } \kappa_r=\min_{(x,y)\in\overline{\Omega}} \varphi_r.$$
For $n\in \mathbb{N}$ and $\sigma_0>0$, let 
$$\sigma=\frac{\alpha \lambda_*^+\sigma_0}{2(\lambda_*^+ -\lambda_*)},\ \sigma_n=\sigma_0 + n\sigma,\ B_n=e^{-\alpha \lambda_*^+ (n-1)\sigma} \hbox{ and } \overline{B}:=\sum_{i=0}^{+\infty} B_i<+\infty.$$
For $n\in\mathbb{N}$ and $n\ge 1$, define 
$$\theta_{\sigma_n}=\frac{\delta (3\sum_{i=0}^{n-1} B_i)^{1+\alpha} }{(k(r)+c^*r) \kappa_r}e^{-[(1+\alpha)\lambda_*^+ -r]\sigma_{n-1}} \hbox{ and }
 \widetilde{\theta}_{\sigma_{n}}=\frac{\theta_{\sigma_{n}}}{\kappa}e^{-(r-\lambda_*)\sigma_{n-1}},$$
where $\lambda_*< r:=\min\{(1+\alpha)\lambda_*,(\lambda_*+\lambda_*^+)/2\}<(1+\alpha)\lambda_*$ and $\delta$ is defined by \eqref{eq-F}. 
Notice that $\sigma_n=\sigma_0 +n\sigma$, $\theta_{\sigma_n}\rightarrow 0$ and $\widetilde{\theta}_{\sigma_n}\rightarrow 0$ as $\sigma_0\rightarrow +\infty$. Take $\sigma_0>0$ large enough such that
\be\label{sigma0}
\overline{B}e^{-\lambda_*^+ \sigma_{0}} \varphi_*^+(x,y)\le \gamma/3,\ \frac{3\sigma \delta (3\overline{B})^{1+\alpha}}{(k(r)+c^* r)\kappa_r \kappa^2} \max\{1,(1+\alpha)\lambda_*^+ -r, r-\lambda_* \} e^{-\frac{1}{2}\alpha\lambda_*^+ \sigma_0} \le 1
\ee
where $\gamma$ is defined by \eqref{eq-F} and
\be\label{theta-alpha}
\delta 3^{1+\alpha} \frac{\theta_{\sigma_n}^{\alpha}}{\kappa^{1+\alpha}} e^{-\alpha r\sigma_{n-1}}\le (k(r)+c^*r) \kappa_r.
\ee
By definitions of $\theta_{\sigma_n}$ and $\widetilde{\theta}_{\sigma_n}$, it follows from \eqref{sigma0} that
\be\label{sigma} 
\begin{aligned}
\widetilde{\theta}_{\sigma_n} e^{-\lambda_* \sigma_{n-1}} \varphi_*(x,y)=&\frac{1}{\kappa} e^{-(r-\lambda_*)\sigma_{n-1}} \frac{\delta (3\sum_{i=0}^{n-1}B_i)^{1+\alpha}}{(k(r)+c^* r) \kappa_r} e^{-[(1+\alpha)\lambda_*^+ -r] \sigma_{n-1}} e^{-\lambda_* \sigma_{n-1}}\varphi_*(x,y)\\
<& \frac{\delta (3\sum_{i=0}^{n-1}B_i)^{1+\alpha}}{(k(r)+c^* r)\kappa_r \kappa^2} e^{-(1+\alpha)\lambda_*^+ \sigma_{n-1}} \psi_*^+(x,y), \\
<& \sum_{i=0}^{n-1}B_i e^{-\lambda_*^+ \sigma_{n-1}} \varphi_*^+(x,y)\le \gamma/3,\quad
\hbox{ for all $(x,y)\in\overline{\Omega}$ and $n\ge 1$}.
\end{aligned}
\ee
By the definition of $\widetilde{\theta}_{\sigma}$, one also has that
\be\label{w-theta}
\widetilde{\theta}_{\sigma_{n}} e^{-\lambda_* s}\varphi_*(x,y)\ge \theta_{\sigma_n} e^{-r s}\varphi_r(x,y), \hbox{ for all $s\ge \sigma_{n-1}$ and $(x,y)\in \overline{\Omega}$}.
\ee

For $n\ge 1$ and $(t,x,y)\in\R\times\overline{\Omega}$, define $s=x\cdot e-c^*t$ and 
$$\overline{u}_{\sigma_n}(t,x,y)=\sum_{i=0}^{n-1}B_i e^{-\lambda_*^+ s}\varphi_*^+(x,y)+  \widetilde{\theta}_{\sigma_n} e^{-\lambda_* s}\varphi_*(x,y)+\theta_{\sigma_n} e^{-r s}\varphi_r(x,y).$$
Then, by \eqref{sigma}-\eqref{w-theta}, one has that $\overline{u}_{\sigma_n}(t,x,y)\le \gamma$ for $s\ge  \sigma_{n-1}$ and $(x,y)\in \overline{\Omega}$.
One can check from \eqref{eq-F} and $k(\lambda_*^+)+c^* \lambda_*^+\ge 0$ that for $s\ge \sigma_{n-1}$ and $(x,y)\in\overline{\Omega}$,
\begin{align*}
\mathscr{L}\overline{u}_{\sigma_n}:=& (\overline{u}_{\sigma_n})_t-\nabla \cdot (A(z)\nabla \overline{u}_{\sigma_n})+q(z)\cdot \nabla \overline{u}_{\sigma_n} -f(z,\overline{u}_{\sigma_n})\\
=& (k(\lambda_*^+)+c^* \lambda_*^+) e^{-\lambda_*^+ s}\varphi_*^+(x,y)+(k(r)+c^* r)\theta_{\sigma_n} e^{-r s}\varphi_r(x,y) +f_u(x,y,0)\overline{u}-f(x,y,\overline{u})\\
\ge& (k(r)+c^* r)\theta_{\sigma_n} e^{-r s}\varphi_r(x,y)  -\delta \Big(\sum_{i=0}^{n-1}B_i e^{-\lambda_*^+ s}\varphi_*^+  +  \widetilde{\theta}_{\sigma_n} e^{-\lambda_* s}\varphi_*+\theta_{\sigma_n} e^{-r s}\varphi_r\Big)^{1+\alpha},
\end{align*}
and $\nu A\nabla\overline{u}_n=0$ on $\R\times\partial\Omega$. For $s\ge \sigma_{n-1}$ and $(x,y)\in \overline{\Omega}$ such that 
$$\sum_{i=0}^{n-1}B_i e^{-\lambda_*^+ s} \varphi_*^+(x,y)\ge \widetilde{\theta}_{\sigma_n} e^{-\lambda_* s} \varphi_*(x,y)\ge \theta_{\sigma_n} e^{-r s} \varphi_r(x,y),$$
it follows from the definition of $\theta_{\sigma_n}$ that 
\begin{align*}
\mathscr{L}\overline{u}_{\sigma_n}\ge& (k(r)+c^*r) \theta_{\sigma_n} e^{-r s} \varphi_r(x,y)-\delta \Big(3\sum_{i=0}^{n-1}B_i \Big)^{1+\alpha} e^{-(1+\alpha)\lambda_*^+ s} (\varphi_*^+)^{1+\alpha}\\
\ge& \delta \Big(3\sum_{i=0}^{n-1}B_i \Big)^{1+\alpha} e^{[(1+\alpha)\lambda_*^+ -r](s-\sigma_{n-1})} e^{-(1+\alpha)\lambda_*^+ s} -\delta \Big(3\sum_{i=0}^{n-1}B_i \Big)^{1+\alpha} e^{-(1+\alpha)\lambda_*^+ s}\\
\ge& 0.
\end{align*}
For $s\ge \sigma_{n-1}$ and $(x,y)\in\overline{\Omega}$ such that $\widetilde{\theta}_{\sigma_n} e^{-\lambda_* s} \varphi_*(x,y)\ge \sum_{i=0}^{n-1}B_ie^{-\lambda_*^+ s} \varphi_*^+(x,y)$, it follows from \eqref{theta-alpha} that
\begin{align*}
\mathscr{L}\overline{u}_{\sigma_n}\ge& (k(r)+c^*r) \theta_{\sigma_n} e^{-r s} \varphi_r(x,y)-\delta 3^{1+\alpha} \widetilde{\theta}_{\sigma_n}^{1+\alpha} e^{-(1+\alpha)\lambda_* s} (\varphi_*)^{1+\alpha}\\
\ge& (k(r)+c^*r) \theta_{\sigma} e^{-r s} \kappa_r -\delta 3^{1+\alpha} \frac{\theta_{\sigma_n}^{1+\alpha}}{\kappa^{1+\alpha}} e^{-\alpha r \sigma_{n-1}} e^{-[(1+\alpha)\lambda_*-r] (s-\sigma_{n-1})} e^{-r s}\\
\ge& 0.
\end{align*}
In conclusion, one has that $\mathscr{L}\overline{u}_{\sigma_n}\ge 0$ for all $(t,x,y)\in\R\times\overline{\Omega}$ such that $s\ge \sigma_{n-1}$ and $n\ge 1$.

Since $\phi(+\infty,x,y)=0$, one can take a constant $\tau$ such that
\be\label{n=0}
\phi(s+\tau,x,y)< e^{-\lambda_*^+ \sigma_0} \psi_*^+(x,y), \hbox{ for all $s\ge \sigma_0$ and $(x,y)\in\overline{\Omega}$}.
\ee
Assume that 
\be\label{0sn-1}
\phi(s+\tau,x,y)\le \overline{u}_{\sigma_n}(t,x,y), \hbox{ for $\sigma_0\le s\le \sigma_{n-1}$ and $(x,y)\in\overline{\Omega}$},
\ee
which also means that 
$$\phi(s+\tau,x,y)\le \overline{u}_{\sigma_n}(t,x,y)|_{s=\sigma_{n-1}}, \hbox{ for all $s\ge \sigma_{n-1}$ and $(x,y)\in\overline{\Omega}$},$$
since $\phi(s,x,y)$ is decreasing in $s$.
Then, by \eqref{n=0}, above inequality holds for $n=1$.
Since $\widetilde{\theta}_{\sigma_n}>0$, it follows from \eqref{s-epsilon} that there is $s_{\sigma_n}> \sigma_{n-1}+\sigma=\sigma_n$ such that
$$\phi(s_{\sigma_n}+\tau,x,y)\le  \widetilde{\theta}_{\sigma_n} e^{-\lambda_* s_{\sigma_n}}\varphi_*(x,y)<\overline{u}_{\sigma_n}(t,x,y), \hbox{ for $\sigma_{n-1}\le s\le s_{\sigma_n}$ and $(x,y)\in\overline{\Omega}$}.$$
By Lemma~\ref{comparison}, it follows that
\be\label{n-1sn}
\phi(s+\tau,x,y)\le \overline{u}_{\sigma_n}(t,x,y), \hbox{ for $(t,x,y)\in\R\times\overline{\Omega}$ such that $\sigma_{n-1}\le s\le s_{\sigma_n}$}.
\ee
Notice that 
\begin{align*}
\theta_{\sigma_{n}}-\theta_{\sigma_{n+1}}=&\frac{\delta (3\sum_{i=0}^{n-1} B_i)^{1+\alpha} }{(k(r)+c^*r) \kappa_r}e^{-[(1+\alpha)\lambda_*^+ -r]\sigma_{n-1}} -\frac{\delta (3\sum_{i=0}^{n} B_i)^{1+\alpha} }{(k(r)+c^*r) \kappa_r}e^{-[(1+\alpha)\lambda_*^+ -r]\sigma_{n}}\\
\le& \frac{\sigma\delta (3\overline{B})^{1+\alpha} }{(k(r)+c^*r) \kappa_r} [(1+\alpha)\lambda_*^+ -r] e^{-[(1+\alpha)\lambda_*^+ -r]\sigma_{n-1}}
\end{align*}
and
\begin{align*}
\widetilde{\theta}_{\sigma_{n}}-\widetilde{\theta}_{\sigma_{n+1}}=& \frac{\theta_{\sigma_{n}}}{\kappa}e^{-(r-\lambda_*)\sigma_{n-1}}- \frac{\theta_{\sigma_{n+1}}}{\kappa}e^{-(r-\lambda_*)\sigma_{n}}\\
\le& \frac{\theta_{\sigma_{n}}-\theta_{\sigma_{n+1}}}{\kappa} e^{-(r-\lambda_*)\sigma_{n-1}} +\frac{\theta_{\sigma_n}}{\kappa}(e^{-(r-\lambda_*)\sigma_{n-1}} -e^{-(r-\lambda_*)\sigma_n})\\
\le&  \frac{\sigma\delta (3\overline{B})^{1+\alpha} }{(k(r)+c^*r) \kappa_r\kappa} [(1+\alpha)\lambda_*^+ -r] e^{-[(1+\alpha)\lambda_*^+ -\lambda_*]\sigma_{n-1}} \\
&+\frac{\sigma\delta (3\overline{B})^{1+\alpha} }{(k(r)+c^*r) \kappa_r\kappa} (r-\lambda_*) e^{-(1+\alpha)\lambda_*^+ \sigma_{n} +\lambda_* \sigma_{n-1} -r\sigma}
\end{align*}
Thus, for $\sigma_0\le s\le \sigma_n$, it follows from the definitions of $\sigma$, $B_n$ and \eqref{sigma0} that
\begin{align*}
\overline{u}_{\sigma_{n+1}}(t,x,y)=&\overline{u}_{\sigma_{n}}(t,x,y) +B_{n} e^{-\lambda_*^+ s} \varphi_*^+ + (\widetilde{\theta}_{\sigma_{n+1}}-\widetilde{\theta}_{\sigma_{n}}) e^{-\lambda_* s}\varphi_*+(\theta_{\sigma_{n+1}}-\theta_{\sigma_n}) e^{-r s}\varphi_r\\
\ge& \overline{u}_{\sigma_{n}}(t,x,y) +e^{-\lambda_*^+ s} \Big( B_n \kappa -(\widetilde{\theta}_{\sigma_{n}}-\widetilde{\theta}_{\sigma_{n+1}}) e^{(\lambda_*^+ -\lambda_*)\sigma_n} - (\theta_{\sigma_{n}}-\theta_{\sigma_{n+1}}) e^{(\lambda_*^+ -r)\sigma_n} \Big)\\
\ge& \overline{u}_{\sigma_{n}}(t,x,y) +e^{-\lambda_*^+ s} \Big( B_n \kappa - 3\frac{\sigma\delta (3\overline{B})^{1+\alpha} }{(k(r)+c^*r) \kappa_r\kappa} \max\{(1+\alpha)\lambda_*^+ -r,r-\lambda_*\}\\
& e^{-\frac{1}{2}\alpha\lambda_*^+ \sigma_0 -\alpha\lambda_*^+(n-1)\sigma} \Big)\\
\ge& \overline{u}_{\sigma_{n}}(t,x,y).
\end{align*}
Then, by \eqref{0sn-1} and \eqref{n-1sn}, one has that
$$\phi(s+\tau,x,y)\le \overline{u}_{\sigma_{n+1}}(t,x,y), \hbox{ for $\sigma_0\le s\le \sigma_{n}$ and $(x,y)\in\overline{\Omega}$}.$$
By iteration, the inequality \eqref{0sn-1} holds for all $n\in\mathbb{N}$ and $n\ge 1$. By passing $n$ to $+\infty$ in \eqref{0sn-1}, one finally has that
\be\label{upbound}
\phi(s+\tau,x,y)\le \overline{B} e^{-\lambda_*^+ s} \varphi_*^+(x,y), \hbox{ for $s\ge  \sigma_0$ and $(x,y)\in\overline{\Omega}$}.
\ee

If $F_{c^*}$ is a singleton, that is, $\{\lambda_*\}$, then it follows from Proposition~2.2 of \cite{H2008} that
$$\lim_{s\rightarrow +\infty} \frac{-\phi_s(s,x,y)}{\phi(s,x,y)}=\lambda_*.$$
This implies  $\phi(s,x,y)\ge C e^{-(\lambda_* +\varepsilon) s}$ as $s\rightarrow+\infty$ for some $\varepsilon>0$ and $C>0$ which contradicts  \eqref{upbound}. Thus, there is $\lambda_*^+>\lambda_*$ in $F_{c^*}$ and \eqref{upbound} still holds.

{\it Step~2: the lower bound.} We need to adjust some notions. For $n\in \mathbb{N}$ and $\sigma_0>0$, let 
$$\sigma=\frac{\alpha \lambda_*^+\sigma_0}{2(\lambda_*^+ - r)},\ \sigma_n=\sigma_0 + n\sigma.$$ 
Let
$$B_0=0,\ B_n=e^{-\alpha \lambda_*^+ (n-1)\sigma} \hbox{ for $n\ge 1$ and } \underline{B}:=1-\sum_{i=0}^{+\infty} B_i.$$
Take $\sigma_0>0$ large enough such that $\underline{B}>0$, 
\be\label{sigma0-u}
e^{-\lambda_*^+ \sigma_{0}} \varphi_*^+(x,y)\le \gamma,\ \frac{\delta}{(k(r)+c^* r)\kappa_r \kappa} \max\{(1+\alpha)e^{-\alpha\lambda_*^+ \sigma_0},\sigma[(1+\alpha)\lambda_*^+ -r] e^{-\frac{1}{2}\alpha\lambda_*^+ \sigma_0}\} \le \frac{1}{2}.
\ee
For $n\in\mathbb{N}$ and $n\ge 1$, define 
$$\theta_{\sigma_n}=\frac{\delta (1-\sum_{i=0}^{n-1} B_i)^{1+\alpha} }{(k(r)+c^*r) \kappa_r}e^{-[(1+\alpha)\lambda_*^+ -r]\sigma_{n-1}}.$$
For $t\in\R$ and $(x,y)\in\overline{\Omega}$, define $s=x\cdot e-c^*t$ and 
$$\underline{u}_{\sigma_n}(t,x,y)=\max\Big\{(1-\sum_{i=0}^{n-1} B_i)e^{-\lambda_*^+ s}\varphi_*^+(x,y)-\theta_{\sigma_n} e^{-r s}\varphi_r(x,y),0\Big\}.$$
Then, by \eqref{sigma0-u}, one has that $\underline{u}_{\sigma_n}(t,x,y)\le \gamma$ for $s\ge  \sigma_{n-1}$ and $(x,y)\in \overline{\Omega}$.
For $n\ge 1$ and $(t,x,y)\in \R\times\overline{\Omega}$ such that $s\ge \sigma_{n-1}$ and $\underline{u}_{\sigma_n}(t,x,y)>0$, one can check from the definition of $\theta_{\sigma_n}$ and \eqref{eq-F} that
\begin{align*}
\mathscr{L}\underline{u}_{\sigma_n}:=& (\underline{u}_{\sigma_n})_t-\nabla \cdot (A(z)\nabla \underline{u}_{\sigma_n})+q(z)\cdot \nabla \underline{u}_{\sigma_n} -f(z,\underline{u}_{\sigma_n})\\
=& -(k(r)+c^* r)\theta_{\sigma_n} e^{-r s}\varphi_r(x,y) +f_u(x,y,0)\underline{u}_{\sigma_n}-f(x,y,\underline{u}_{\sigma_n})\\
\le& -(k(r)+c^* r)\theta_{\sigma_n} e^{-r s}\varphi_r(x,y)  +\delta \underline{u}_{\sigma_n}^{1+\alpha}\\
\le& -(k(r)+c^* r)\theta_{\sigma_n} e^{-r s}\varphi_r(x,y)  +\delta (1-\sum_{i=0}^{n-1} B_i)^{1+\alpha}e^{-(1+\alpha)\lambda_*^+ s}(\varphi_*^+)^{1+\alpha}\\
\le& 0,
\end{align*}
and $\nu A\nabla\underline{u}_n=0$ on $\R\times\partial\Omega$.

Since $\phi(-\infty,x,y)=1$, one can take a constant $\eta$ such that
\be\label{n0-u} 
\phi(\sigma_0-\eta,x,y)\ge e^{-\lambda_*^+ \sigma_0} \varphi_*^+(x,y).
\ee
Assume that 
\be\label{u-n-1}
\phi(s-\eta,x,y)\ge \underline{u}_{\sigma_n}(t,x,y), \hbox{ for $\sigma_0\le s\le \sigma_{n-1}$ and $(x,y)\in \overline{\Omega}$},
\ee
which also means that
$$\phi(\sigma_{n-1}-\eta,x,y)\ge \underline{u}_{\sigma_n}(t,x,y), \hbox{ for $s\ge \sigma_{n-1}$ and $(x,y)\in \overline{\Omega}$},$$
since $\underline{u}_{\sigma_n}(t,x,y)$ is decreasing in $s$.
Then, by \eqref{n0-u}, above inequality holds for $n=1$. Since $r<\lambda_*^+$, there is $s_{\sigma_n}>\sigma_{n-1}+\sigma=\sigma_n$ such that
$$
\underline{u}_{\sigma_n}(t,x,y)|_{s=s_{\sigma_n}}=0\le \phi(s-\eta,x,y), \hbox{ for $\sigma_{n-1}\le s\le s_{\sigma_n}$ and $(x,y)\in\overline{\Omega}$}.
$$
Thus, by Lemma~\ref{comparison}, it follows that
\be\label{u-n-1i}
\phi(s-\eta,x,y)\ge \underline{u}_{\sigma_n}(t,x,y), \hbox{ for $\sigma_0\le s\le s_{\sigma_n}$ and $(x,y)\in\overline{\Omega}$}.
\ee
Notice that 
\begin{align*}
\theta_{\sigma_{n}}-\theta_{\sigma_{n+1}}=&\frac{\delta (1-\sum_{i=0}^{n-1} B_i)^{1+\alpha} }{(k(r)+c^*r) \kappa_r}e^{-[(1+\alpha)\lambda_*^+ -r]\sigma_{n-1}} -\frac{\delta (1-\sum_{i=0}^{n} B_i)^{1+\alpha} }{(k(r)+c^*r) \kappa_r}e^{-[(1+\alpha)\lambda_*^+ -r]\sigma_{n}}\\
\le& \frac{\sigma\delta }{(k(r)+c^*r) \kappa_r} [(1+\alpha)\lambda_*^+ -r] e^{-[(1+\alpha)\lambda_*^+ -r]\sigma_{n-1}} +\frac{\delta(1+\alpha) B_{n}}{(k(r)+c^*r) \kappa_r} e^{-[(1+\alpha)\lambda_*^+ -r]\sigma_{n}}.
\end{align*}
Thus, for $\sigma_0\le s\le \sigma_n$, it follows from the definitions of $\sigma$, $B_n$ and \eqref{sigma0-u} that 
\begin{align*}
\underline{u}_{\sigma_{n+1}}(t,x,y)=&\underline{u}_{\sigma_{n}}(t,x,y) -B_{n} e^{-\lambda_*^+ s} \varphi_*^+ +(\theta_{\sigma_{n}}-\theta_{\sigma_{n+1}}) e^{-r s}\varphi_r\\
\le& \underline{u}_{\sigma_{n}}(t,x,y) +e^{-\lambda_*^+ s} (-B_n \kappa +(\theta_{\sigma_{n}}-\theta_{\sigma_{n+1}}) e^{(\lambda_*^+-r)\sigma_n})\\
\le& \underline{u}_{\sigma_n}(t,x,y) + e^{-\lambda_*^+ s}  (-B_n \kappa +\frac{\sigma\delta }{(k(r)+c^*r) \kappa_r} [(1+\alpha)\lambda_*^+ -r] e^{-\alpha\lambda_*^+\sigma_{n-1} +(\lambda_*^+-r)\sigma}\\
& +\frac{\delta(1+\alpha) B_{n}}{(k(r)+c^*r) \kappa_r} e^{-\alpha\lambda_*^+ \sigma_{n}})\\
\le& \underline{u}_{\sigma_n}.
\end{align*}
By \eqref{u-n-1} and \eqref{u-n-1i}, one has that
$$\phi(s-\eta,x,y)\ge \underline{u}_{\sigma_{n+1}}(t,x,y), \hbox{ for $\sigma_0\le s\le \sigma_{n}$ and $(x,y)\in \overline{\Omega}$}.$$
By iteration, the inequality \eqref{u-n-1} holds for all $n\in\mathbb{N}$ and $n\ge 1$. By passing $n$ to $+\infty$ in \eqref{u-n-1}, one finally has that
$$\phi(s-\eta,x,y)\ge \underline{B} e^{-\lambda_*^+ s} \varphi_*^+(x,y), \hbox{ for $s\ge \sigma_0$ and $(x,y)\in\overline{\Omega}$}.$$

{\it Step 3: convergence to $e^{-\lambda_*^+ s}\varphi_*^+(x,y)$.} By Step~2 and Step~3, one has
$$0<B:=\liminf_{s\rightarrow +\infty}\Big(\min_{(x,y)\in\overline{\Omega}}\frac{\phi(s,x,y)}{e^{-\lambda_*^+ s}\varphi_*^+(x,y)}\Big)\le \limsup_{s\rightarrow +\infty}\Big(\max_{(x,y)\in\overline{\Omega}}\frac{\phi(s,x,y)}{e^{-\lambda_*^+ s}\varphi_*^+(x,y)}\Big):=B'<+\infty.$$
By the definition of $B'$,  there is a sequence $(s_n,x_n,y_n)_{n\in\mathbb{N}}$ in $\R\times\overline{\Omega}$ such that
$$s_n\rightarrow +\infty \hbox{ and } \frac{\phi(s_n,x_n,y_n)}{e^{-\lambda_*^+ s_n}\varphi_*^+(x_n,y_n)}\rightarrow B' \hbox{ as $n\rightarrow +\infty$}.$$
By the same argument as in the proof of Theorem~1.3 of \cite{H2008}, one can get that
$$\frac{\phi(s_n,x,y)}{e^{-\lambda_*^+ s_n}\varphi_*^+(x,y)}\rightarrow B' \hbox{ as $n\rightarrow +\infty$, uniformly in $(x,y)\in \overline{\Omega}$}.
$$
Assume that $B'>B$. Then, for $n$ large enough, one has that
$$\phi(s_n,x,y)\ge \frac{B+B'}{2} e^{-\lambda_*^+ s_n} \varphi_*^+(x,y) \hbox{ for all $(x,y)\in \overline{\Omega}$}.$$
Now, we can do the same argument as in Step~3 by replacing $\sigma_0$ by $s_n$ to get that
$$\phi(s,x,y)\ge (\frac{B+B'}{2}-1+\underline{B}) e^{-\lambda_*^+ s} \varphi_*^+(x,y), \hbox{ for $s\ge s_n$ and $(x,y)\in \overline{\Omega}$}.$$
Notice that $\underline{B}\rightarrow 1$ as $\sigma_0\rightarrow +\infty$. For $n$ large enough, one has $(B+B')/2-1+\underline{B}>B$.
This contradicts the definition of $B$. Therefore, $B=B'$.

The proof is thereby complete.
\end{proof}

\subsection{Proof of Theorem~\ref{Th2}}

\begin{proof}[Proof of Theorem~\ref{Th2}]
Let $\lambda_c=\min F_c$ and $\lambda_c^+$ be equal to $\max F_c$ if $F_c$ is of two points, to a large positive constant if $F_c$ is a singleton. Proposition~4.3 of \cite{H2008} has shown us that
$$ \limsup_{s\rightarrow +\infty}\Big[\max_{(x,y)\in \overline{\Omega}}\Big(\frac{\phi(s,x,y)}{e^{-\lambda_c s}}\Big)\Big]<+\infty.$$ Assume that $\phi(s,x,y)$ does not decay as the rate $e^{-\lambda_c s}$. Then, by following Step~2 of the proof of Theorem~\ref{Th1}, there is $\overline{B}>0$ such that 
$$\phi(s,x,y) \le \overline{B} e^{-\lambda_c^+ s} \varphi_{\lambda_c^+}(x,y), \hbox{ for $s$ large enough and $(x,y)\in\overline{\Omega}$}.$$
This actually contradicts Theorem~1.5 of \cite{H2008}, that is, $\ln \phi(s,x,y)\sim -\lambda_c s$ as $s\rightarrow+\infty$ and $(x,y)\in\overline{\Omega}$. Thus, $\phi(s,x,y)$ satisfies 
$$0<\liminf_{s\rightarrow +\infty} \Big[\min_{(x,y)\in \overline{\Omega}}\Big(\frac{\phi(s,x,y)}{e^{-\lambda_c s}}\Big)\Big]\le \limsup_{s\rightarrow +\infty}\Big[\max_{(x,y)\in \overline{\Omega}}\Big(\frac{\phi(s,x,y)}{e^{-\lambda_c s}}\Big)\Big]<+\infty.$$
By the proof of Theorem~1.3 of \cite{H2008}, one can get that there exists $B>0$ such that 
$$
\phi(s,x,y)\sim B e^{-\lambda_c s}\varphi_{\lambda_c}(x,y), \hbox{ as $s\rightarrow +\infty$ uniformly in $(x,y)\in\overline{\Omega}$}.
$$
This completes the proof.
\end{proof}

\section{Stability of the pushed front}

This section is devoted to the proof of stability of the pushed front. We first prove a Liouville type result.

\begin{proposition}\label{Lio}
Assume that $v(t,x,y)$ is a solution of \eqref{RD} satisfying
\be\label{trapv}
\phi(x -c^* t-\tau_1,x,y)\le v(t,x,y)\le \phi(x-c^* t-\tau_2,x,y),
\ee
for some $\tau_1<\tau_2$. Then, there is $\tau\in [\tau_1,\tau_2]$ such that
$$v(t,x,y)=\phi(x -c^* t+\tau,x,y), \hbox{ for all $(t,x,y)\in \R\times\overline{\Omega}$}.$$
\end{proposition}

\begin{proof}
By \eqref{trapv} and Theorem~\ref{Th1}, one knows that
$$0<D:=\liminf_{s\rightarrow +\infty}\Big(\min_{(x,y)\in\overline{\Omega}}\frac{v(t,x,y)}{e^{-\lambda_*^+ s}\varphi_*^+(x,y)}\Big)\le \limsup_{s\rightarrow +\infty}\Big(\max_{(x,y)\in\overline{\Omega}}\frac{v(t,x,y)}{e^{-\lambda_*^+ s}\varphi_*^+(x,y)}\Big):=D'<+\infty,$$
where $s=x-c^* t$. By following similar arguments as Step~4 of the proof of Theorem~1.2, one can get that $D=D'$. Thus, 
$$v(t,x,y)\sim D e^{-\lambda_*^+ s} \varphi_*^+(x,y) \hbox{ as $s\rightarrow +\infty$ uniformly in $(x,y)\in\overline{\Omega}$}.$$
Let $\tau_*=\ln(D/B)/\lambda_*^+$ where $B$ is defined in Theorem~\ref{Th1}. This means that 
\be\label{vcphi}
\frac{v(t,x,y)}{\phi(x-c^*t-\tau_*,x,y)}\rightarrow 1 \hbox{ as $x-c^*t\rightarrow +\infty$}.
\ee

Let
$$\tau^*=\max\{\tau\in \R; v(t,x,y)\ge \phi(x-c^* t-\tau',x,y) \hbox{ in $\R\times\overline{\Omega}$ for all $\tau'\le \tau$}\}.$$
Then, $\tau_1\le \tau^*\le \tau_2$ since $\phi_s<0$ and $\tau^*\le \tau_*$ by \eqref{vcphi}. Assume that $\tau^*<\tau_*$. Then, by \eqref{vcphi}, there exists $\sigma_0>0$ such that
\be\label{vcphi2}
v(t,x,y)\ge \phi(x-c^*t-(\tau^* +\tau_*)/2,x,y) \hbox{ for $(t,x,y)\in\R\times\overline{\Omega}$ such that $x-c^*t\ge \sigma_0$}.
\ee

Let $\rho$ be the constant such that
$$|f_u(x,y,1)-f_u(x,y,1-u)|\le \mu_1 \hbox{ for all $(x,y,u)\in \overline{\Omega}\times[0,\rho]$},$$
and let $\rho_{\tau}(t,x,y)=\tau\rho\psi_1(x,y)$ for each $\tau\in [0,1]$ where $\mu_1$ and $\psi_1(x,y)$ are defined by \eqref{mu1}. Then, $\rho_{\tau}(t,x,y)$ satisfies the weak stability condition (1.6) of \cite{H2008}. Take $r\in (\lambda_*,\lambda_*^+)$. Then, $k(r)+c^*r>0$. Since $\phi(-\infty,\cdot,\cdot)=1$ and $\phi(+\infty,\cdot,\cdot)=0$, there is $\sigma_1\ge \sigma_0>0$ such that
$$\phi(s-\tau_1,x,y)\ge 1-\rho, \hbox{ for $s\le -\sigma_1$ and $(x,y)\in\overline{\Omega}$},$$
and 
\be\label{tau2} 
\phi(s-\tau_2,x,y)\le \min\Big\{\gamma,(\frac{k(r)+c^*r}{\delta})^{\frac{1}{1+\alpha}}\Big\}, \hbox{ for $s\ge \sigma_1$ and $(x,y)\in \overline{\Omega}$},
\ee
where $\delta$ and $\gamma$ are defined in \eqref{eq-F}.

Assume that
$$\inf_{(t,x,y)\in \R\times\overline{\Omega},\ -\sigma_1\le x\cdot e-c^*t\le \sigma_1} \{v(t,x,y)-\phi(x\cdot e-c^*t-\tau^*,x,y)\}>0.$$
Then, by continuity, there is $0<\eta^*\le \tau_2-\tau^*$ such that for any $\eta\in [0,\eta^*]$,
$$\phi(x\cdot e-c^*t-(\tau^* +\eta),x,y)\le v(t,x,y), \hbox{ for $-\sigma_1\le x\cdot e-c^* t\le \sigma_1$}.$$
For any $\varepsilon>0$ and $(t,x,y)\in \R\times\overline{\Omega}$ such that $x-c^* t\ge \sigma_1$, define $s=x-c^* t$ and
$$\underline{v}(t,x,y)=\max\{\phi(x\cdot e-c^*t-(\tau^* +\eta),x,y)-\varepsilon e^{-r s} \varphi_r(x,y),0\}.$$
Since $r<\lambda_*^+$ and $\phi\sim Be^{-\lambda_*^+ s} \varphi_*^+(x,y)$ as $s\rightarrow +\infty$, there is $s_{\varepsilon}$ such that $s_{\varepsilon}\rightarrow +\infty$ as $\varepsilon\rightarrow 0$ and 
$$\underline{v}(t,x,y)=0\le v(t,x,y) \hbox{ for $x-c^*t \ge s_{\varepsilon}$}.$$
For $(t,x,y)\in \R\times\overline{\Omega}$ such that $x-c^*t\ge \sigma_1$ and $\underline{v}(t,x,y)>0$, one can check from \eqref{eq-F} and \eqref{tau2} that
\begin{align*}
\mathscr{L}\underline{v}:=& (\underline{v})_t-\nabla \cdot (A(z)\nabla \underline{v})+q(z)\cdot \nabla \underline{v} -f(z,\underline{v})\\
=& -(k(r)+c^* r)\varepsilon e^{-r s}\varphi_r(x,y) - f_u(x,y,0)\varepsilon e^{-r s}\varphi_r(x,y) +f(x,y,\phi)-f(x,y,\underline{v})\\
\le& -(k(r)+c^* r)\varepsilon e^{-r s}\varphi_r(x,y)+\varepsilon e^{-r s} \varphi_r(x,y)\delta \phi^{1+\alpha}\\
\le& 0,
\end{align*}
and $\nu A \nabla\underline{v}=0$ on $\R\times\partial\Omega$.
By the comparison principle, one has that
$$\underline{v}(t,x,y)\le v(t,x,y) \hbox{ for $\sigma_0\le x\cdot e-c^* t\le s_{\varepsilon}$}.$$
By passing $\varepsilon$ to $0$, one gets that
$$\phi(x-c^*t -(\tau^*+\eta),x,y)\le v(t,x,y), \hbox{ for $x\cdot e-c^* t\ge \sigma_1$}.$$
In the region where $x-c^* t\le -\sigma_1$,  we have $\phi(x -c^*t-(\tau^*+\eta),x,y)>1-\rho$. All assumptions of Lemma~2.1 of \cite{HR2011} are satisfied with $h=-\sigma_1$, $\overline{U}(t,x,y)=\phi(x-c^* t-(\tau^*+\eta),x,y)$, $\overline{\Phi}=\phi(\cdot -(\tau^*+\eta),\cdot,\cdot)$, $\underline{U}=v$ and $\underline{\Phi}=v((x-s)/c^*,x,y)$, apart from the fact that $\underline{\Phi}$ may not be periodic in $(x,y)$. But the arguments used in the proof of Lemma~2.1 of \cite{HR2011}, that is, Lemma~2.3 of \cite{H2008} can be immediately extended to the present case. Then,
$$\phi(x-c^*t -(\tau^*+\eta),x,y)\le v(t,x,y), \hbox{ for $x\cdot e-c^* t\le -\sigma_1$}.$$
Thus, $\phi(x-c^*t -(\tau^*+\eta),x,y)\le v(t,x,y)$  for all $(t,x,y)\in\R\times\overline{\Omega}$ which contradicts the definition of $\tau^*$. 

Therefore,
$$\inf_{(t,x,y)\in \R\times\overline{\Omega},\ -\sigma_1\le x\cdot e-c^*t\le \sigma_1} \{v(t,x,y)-\phi(x-c^*t-\tau^*,x,y)\}=0.$$
Then, there exists a sequence $(t_n,x_n,y_n)_{n\in\mathbb{N}}$ such that $s_n=x_n -ct_n-\tau^*\in [-\sigma_1,\sigma_1]$ for all $n\in\mathbb{N}$, and 
$$v(t_n,x_n,y_n)-\phi(x_n -c^*t_n-\tau^*,x_n,y_n)\rightarrow 0 \hbox{ as $n\rightarrow +\infty$}.$$
The arguments of the last paragraph of p.376 of \cite{HR2011} yield that there exist $s_{\infty}\in\R$ and $x_{\infty}\in\R$ such that $x_n-c^*t_n\rightarrow s_{\infty}$, $x''_n\rightarrow x_{\infty}$ as $n\rightarrow +\infty$ and
\be\label{limv}
v(t+t_n,x+x'_n,y)\rightarrow \phi(x-c^*t+s_{\infty}-x_{\infty}-\tau^*,x,y), \hbox{ as $n\rightarrow +\infty$ locally uniformly in $\overline{\Omega}$},
\ee
where $x_n=x'_n+x''_n$, $x_n'\in L_1\mathbb{Z}$ and $x''_n\in [0,L_1]$. On the other hand, it follows from \eqref{vcphi2} that 
$$v(t+t_n,x+x'_n,y)\ge \phi(x-c^*t-(\tau^*+\tau_*)/2,x,y) \hbox{ for $x-c^*t\ge \sigma_0+c^*t_n-x'_n$},$$
which contradicts \eqref{limv} by passing $n$ to $+\infty$ and $\tau^*<\tau_*$.

Consequently, $\tau^*=\tau_*$ and $v(t,x,y)\ge \phi(x\cdot e-c^*t-\tau_*,x,y)$ in $\R\times\overline{\Omega}$. Similarly, one can get the opposite inequality. Eventually, $v(t,x,y)=\phi(x\cdot e-c^*t-\tau_*,x,y)$.
\end{proof}
\vskip 0.3cm

Then, we need the following lemma.

\begin{lemma}\label{phisphi}
Assume that all assumptions of Theorem~\ref{Th1} are satisfied. Then, it holds that
$$\frac{\phi_s(s,x,y)}{\phi(s,x,y)}\rightarrow -\lambda_*^+,\hbox{ as $s\rightarrow +\infty$}.$$
\end{lemma}

\begin{proof}
Notice that $|\phi_s/\phi|$ is bounded. It means that
$$0<B:=\liminf_{s\rightarrow -\infty}\Big(\min_{(x,y)\in\overline{\Omega}}\frac{\phi_s(s,x,y)}{e^{-\lambda_*^+ s}\varphi_*^+(x,y)}\Big)\le \limsup_{s\rightarrow -\infty}\Big(\max_{(x,y)\in\overline{\Omega}}\frac{\phi_s(s,x,y)}{e^{-\lambda_*^+ s}\varphi_*^+(x,y)}\Big):=B'<+\infty.$$
Follow similar arguments as Step~4 of the proof of Theorem~\ref{Th1}, one can get that $B'=B$. This implies that there is $\lambda$ such that
$$\frac{\phi_s(s,x,y)}{\phi(s,x,y)}\rightarrow \lambda \hbox{ as $s\rightarrow +\infty$}.$$
By the proof of Proposition~2.2 of \cite{H2008}, the constant $\lambda$ should satisfy $k(\lambda)+c^*\lambda=0$ which means $\lambda=\lambda_*$ or $\lambda_*^+$. Since $\phi(s,x,y)\sim Be^{-\lambda_*^+ s}\varphi_*^+(x,y)$ as $s\rightarrow +\infty$, the constant $\lambda$ should be $\lambda_*^+$.
\end{proof}
\vskip 0.3cm

Let $\theta(x,y)$ be a $C^2(\overline{\Omega})$ nonpositive periodic function such that
$$\nu A\nabla\theta +\nu Ae=0 \hbox{ on $\partial\Omega$}.$$
Take a constant $\lambda$ such that 
$$\lambda_*<\lambda<\min\{r,\lambda_*^+\} \hbox{ and } 0<k(\lambda)+c^*\lambda<\mu_1,$$
where $\mu_1$ is defined by \eqref{mu1} and $r$ is given in Theorem~\ref{Th3}. Define $\eta:=(k(\lambda)+c^*\lambda)/2$. Let $s_*\in \R$ such that
$$e^{-\lambda(s_*-1)}\le \min_{\overline{\Omega}} \psi_1,$$
where $\psi_1$ is defined by \eqref{mu1} with $\|\psi_1\|_{L^{\infty}(\mathcal{C})}=1$.
Let $\chi$ be a $C^2(\R; [0,1])$ function such that
$$\chi'(s)\ge 0 \hbox{ for all $s\in \R$},\ \chi(s)=0 \hbox{ for $s\le s_*-1$},\ \chi(s)=1 \hbox{ for $s\ge s_*$}.$$
For $(s,x,y)\in\R\times\overline{\Omega}$, define 
$$g(s,x,y)=\varphi_{\lambda}(x,y) e^{-\lambda s} \chi(s+\theta(x,y)) +\psi_1(x,y)(1-\chi(s+\theta(x,y))).$$
By Lemma~3.1 of \cite{HR2011}, there is $s_0\le 0$ such that 
\be\label{s0}
\phi(s,x,y)-\epsilon g(s+s_0,x,y)\le 1-\frac{\epsilon}{2} \psi_1(x,y), \hbox{ for all $(s,x,y)\in \R\times\overline{\Omega}$ and $\epsilon\in [0,1]$}.
\ee
For $(s,x,y)\in\R\times\overline{\Omega}$, set
\begin{align*}
B(s,x,y)=&(f_u(x,y,0)+\eta) \varphi_{\lambda} e^{-\lambda s} \chi(s+\theta) +(f_u(x,y,1)+\mu_1 -\eta) \psi_1 (1-\chi(s+\theta)) \\
&+\{(\varphi_{\lambda}e^{-\lambda s} -\psi_1)\times[c^*+ q\cdot (\nabla \theta+e)-\nabla\cdot (A\nabla\theta +A e)]\\
&+2(-\lambda\varphi_{\lambda} e^{-\lambda s} e -e^{-\lambda s}\nabla\varphi_{\lambda} +\nabla\psi_1)A(\nabla\theta +e)\}\chi'(s+\theta)\\
& -(\varphi_{\lambda}e^{-\lambda s} -\psi_1)(\nabla\theta A\nabla\theta+eAe+2e A\nabla\theta)\chi''(s+\theta),
\end{align*}
where $e=(1,0,\cdots,0)$ and
\begin{align*}
C(s,x,y)=-\lambda\varphi_{\lambda}e^{-\lambda s}\chi(s+\theta) +(\varphi_{\lambda} e^{-\lambda s} -\psi_1)\chi'(s+\theta).
\end{align*}
Let $\gamma>0$ be a constant such that
\be\label{gamma}
|f_u(x,y,1)-f_u(x,y,u)|\le \eta \hbox{ for $1-\gamma\le u\le 1$ and $0\le u\le \gamma$}.
\ee
Since $\phi(+\infty,\cdot,\cdot)=0$ and by Lemma~\ref{phisphi}, $\lambda<\lambda_*^+$, there is $s_2\ge s_*$ such that
\begin{eqnarray}\label{s2}
\left\{\begin{array}{lll}
&\phi(s,x,y)\le \frac{\gamma}{2},\ \phi_s(s,x,y)\le -\lambda \phi(s,x,y)&\\
&g(s+s_0,x,y)=e^{-\lambda (s+s_0)} \varphi_{\lambda}(x,y),&\\
&B(s+s_0,x,y)=(f_u(x,y,0)+\eta) e^{-\lambda (s+s_0)}\varphi_{\lambda}(x,y),&\\
&C(s+s_0,x,y)=-\lambda e^{-\lambda (s+s_0)}\varphi_{\lambda}(x,y)<0,&
\end{array} 
\hbox{ for $s\ge s_2$ and $(x,y)\in\overline{\Omega}$.}
\right.
\end{eqnarray}
Since $\phi(-\infty,\cdot,\cdot)=1$ and by the definitions of $g$, $B$, $C$, there is $s_1\le s_*-1$ such that
\begin{eqnarray}\label{s1}
\left\{\begin{array}{lll}
&\phi(s,x,y)\ge 1-\frac{\gamma}{2},&\\
&g(s,x,y)=\psi_1(x,y),&\\
&B(s,x,y)=(f_u(x,y,1)+\mu_1 -\eta) \psi_1(x,y),&\\
&C(s,x,y)=0,&
\end{array} 
\hbox{ for $s\le s_1$ and $(x,y)\in\overline{\Omega}$.}
\right.
\end{eqnarray}
Let $k>0$ be the constant such that
\be\label{ck}
-\phi_s(s,x,y)\ge k, \hbox{ for $s_1\le s\le s_2$ and $(x,y)\in\overline{\Omega}$}.
\ee
Notice that all functions $g$, $B$ and $C$ are bounded in $\R\times\overline{\Omega}$.
Define
\be\label{rho}
\rho=\min\Big\{\frac{k}{2\|C\|_{\infty}},\frac{\gamma}{2}, 1\Big\},
\ee
and
\be\label{comega}
\omega=\frac{2(M\|g\|_{\infty}+\|B\|_{\infty})}{k\eta},
\ee
where $M=\sup_{(x,y)\in\overline{\Omega},u\in[0,1]}|f_u(x,y,u)|$. Let $\varepsilon_0=\rho\min_{\overline{\Omega}}\psi_1$.

For $t\ge 0$ and $(x,y)\in\overline{\Omega}$, define
$$\underline{u}(t,x,y)=\max\{\phi(\underline{s}(t,x),x,y)-\rho g(\underline{s}(t,x)+s_0,x,y) e^{-\eta t},0\},$$
and
$$\overline{u}(t,x,y)=\min\{\phi(\overline{s}(t,x),x,y)+\rho g(\overline{s}(t,x),x,y) e^{-\eta t},1\},$$
where $\underline{s}(t,x)=x -c^*t +\omega\rho -\omega\rho e^{-\eta t}+\sigma_0$ and  $\overline{s}(t,x)=x -c^*t -\omega\rho +\omega\rho e^{-\eta t}-\sigma_0$. Here, $\sigma_0$ is a constant to be given.

\begin{lemma}\label{subsup-s}
Under all assumptions of Theorem~\ref{Th3} with above notions, there exists $\sigma_0>0$ such that
$$\underline{u}(t,x,y)\le u(t,x,y)\le \overline{u}(t,x,y), \hbox{ for $t\ge 0$ and $(x,y)\in\overline{\Omega}$}.$$
\end{lemma}

\begin{proof}
{\it Step~1: when $t=0$.} By assumptions of Theorem~\ref{Th3}, there is $\tau>0$ such that 
$$u_0(x,y)\ge 1-\frac{\varepsilon_0}{2}, \hbox{ for $(x,y)\in\overline{\Omega}$ such that $x\le -\tau$}.$$
By the definition of $s_0$, that is, \eqref{s0}, one knows that
$$\underline{u}(0,x,y)\le 1-\frac{\rho}{2}\psi_1(x,y)\le 1-\frac{\varepsilon_0}{2}\le u_0(x,y), \hbox{ for $(x,y)\in\overline{\Omega}$ such that $x\le -\tau$}.$$
Since $g(s+s_0,x,y)=e^{-\lambda (s+s_0)} \varphi_{\lambda}$ for $s\ge s_2$, $\phi(s,x,y)\sim B e^{-\lambda_*^+ s} \varphi_*^+(x,y)$ as $s\rightarrow +\infty$ and $\lambda<\lambda_*^+$, there is $\sigma_0>0$ such that $\phi(s+\sigma_0,x,y)-\rho g(s+\sigma_0+s_0,x,y)\le 0$ for $s\ge -\tau$ and $(x,y)\in\overline{\Omega}$.
Then,
$$\underline{u}(0,x,y)=\max\{\phi(x+\sigma_0,x,y)-\rho g(x+\sigma_0 +s_0,x,y),0\}=0\le u_0(x,y),$$ 
for $(x,y)\in\overline{\Omega}$ such that $x\ge -\tau$. Thus, one has that $\underline{u}(0,x,y)\le u_0(x,y)$ for all $(x,y)\in\overline{\Omega}$.

By the definition of $g(s,x,y)$, there is $\delta>0$ such that 
$$g(s,x,y)\ge \delta,\hbox{ for $s\le s_2$ and $(x,y)\in\overline{\Omega}$}.$$
By assumptions of Theorem~\ref{Th3} and $\lambda<r$, there are $C>0$ and $\tau>0$ such that 
$$u_0(x,y)\le \rho\delta \hbox{ and } u_0(x,y)\le C e^{-r x}=\rho e^{-\lambda x}\varphi_{\lambda}(x,y), \hbox{ for $(x,y)\in\overline{\Omega}$ such that $x\ge \tau$}.$$
Since $g(s,x,y)=\psi_1(x,y)$ for $s\le s_1$ and $(x,y)\in\overline{\Omega}$ and $\phi(-\infty,\cdot,\cdot)=1$, there is $\sigma_0>0$ such that $\phi(s-\sigma_0,x,y)+\rho g(s-\sigma_0,x,y)\ge 1$ for $s\le \tau$ and $(x,y)\in\overline{\Omega}$. Then,
$$\overline{u}(0,x,y)=\min\{\phi(x-\sigma_0,x,y)+\rho g(x-\sigma_0,x,y),1\}=1\ge u_0(x,y),$$
for $(x,y)\in \overline{\Omega}$ such that $x\le \tau$. For $(x,y)\in\overline{\Omega}$ such that $\tau\le x\le s_2+\sigma_0$, one has that
$$\overline{u}(0,x,y)\ge \rho g(x-\sigma_0,x,y)=\rho \delta\ge u_0(x,y).$$
Since $g(s,x,y)=e^{-\lambda s} \varphi_{\lambda}$ for $s\ge s_2$ and $\sigma_0>0$, it follows that
$$\overline{u}(0,x,y)\ge \rho g(x-\sigma_0,x,y)=\rho e^{-\lambda(x-\sigma_0)}\varphi_{\lambda}\ge u_0(x,y), \hbox{ for $(x,y)\in\overline{\Omega}$ such that $x\ge s_2+\sigma_0\ge \tau$}.$$
Thus, one has that $\overline{u}(0,x,y)\ge u_0(x,y)$ for all $(x,y)\in\overline{\Omega}$.

{\it Step~2: $\underline{u}(t,x,y)$ is a subsolution.}
One can easily check that 
$$\nu A\nabla \underline{u}=0 \hbox{ on $(t,x,y)\in [0,+\infty)\times\partial\Omega$}.$$
Define
$$\Omega^-(t,x,y):=\{(t,x,y)\in [0,+\infty)\times\overline\Omega; \underline{u}(t,x,y)>0\}.$$
After a lengthy but straightforward calculation, one can get that
\begin{align*}
\mathscr{L}\underline{u}:=&\underline{u}_t-\nabla\cdot (A\nabla \underline{u}) +q \cdot \nabla \underline{u} -f(x,y,\underline{u})\\
=& \omega\rho\eta e^{-\eta t} \phi_s +f(x,y,\phi)-f(x,y,\underline{u}) -B(\underline{s}+s_0,x,y)\rho e^{-\eta t} -C(\underline{s}+s_0,x,y)\omega\rho^2\eta e^{-2\eta t},
\end{align*}
for $(t,x,y)\in \Omega^-$ where $\phi_s$, $\phi$ are taken values at $(\underline{s},x,y)$, $\underline{u}$ is taken values at $(t,x,y)$ and $\underline{s}$ denotes $\underline{s}(t,x)$.
If $(t,x,y)\in \Omega^-$ such that $\underline{s}(t,x)\ge s_2$, it follows from \eqref{s2} that
$$\underline{u}(t,x,y)\le \phi(\underline{s},x,y)\le \frac{\gamma}{2}.$$
Then, by \eqref{gamma} and \eqref{s2}, one has that
$$f(x,y,\phi)-f(x,y,\underline{u})\le \rho e^{-\lambda (\underline{s}+s_0)} \varphi_{\lambda} e^{-\eta t} (f_u(x,y,0)+\eta).$$
Notice that $\underline{u}(t,x,y)>0$ means 
\be\label{undu0}
\phi(\underline{s},x,y)>\rho g(\underline{s}+s_0,x,y)=\rho e^{-\lambda(\underline{s}+s_0)} \varphi_{\lambda}(x,y),
\ee
for $(t,x,y)\in\Omega^-$ such that $\underline{s}(t,x)\ge s_2$ by \eqref{s2}.
Thus, it follows from $\phi_s<0$, \eqref{s2} and \eqref{undu0} that
\begin{align*}
\mathscr{L}\underline{u}\le& \rho e^{-\lambda (\underline{s}+s_0)} \varphi_{\lambda} e^{-\eta t} (f_u(x,y,0)+\eta) -(f_u(x,y,0)+\eta) \rho e^{-\lambda (\underline{s}+s_0)} \varphi_{\lambda} e^{-\eta t} +\omega\rho\eta e^{-\eta t} \phi_s \\
&+\lambda e^{-\lambda (\underline{s}+s_0)} \varphi_{\lambda}\omega\rho^2\eta e^{-2\eta t}\\
\le & \omega\rho\eta e^{-\eta t} (-\lambda\phi +\rho\lambda e^{-\lambda (\underline{s}+s_0)} \varphi_{\lambda})\\
\le& 0.
\end{align*}
If $(t,x,y)\in \Omega^-$ such that $\underline{s}(t,x)\le s_1$, it follows from \eqref{s1} and \eqref{rho} that
$$\phi(\underline{s},x,y)\ge \underline{u}(t,x,y)\ge \phi(\underline{s},x,y)-\rho g(\underline{s}+s_0,x,y)\ge 1-\gamma.$$
Then, by \eqref{gamma} and \eqref{s1}, one has that
$$f(x,y,\phi)-f(x,y,\underline{u})\le \rho \psi_1(x,y) e^{-\eta t}(f_u(x,y,1)+\eta).$$
Thus, it follows from $\phi_s<$0 and \eqref{s1} that
$$\mathscr{L}\underline{u}\le \rho \psi_1(x,y) e^{-\eta t}(f_u(x,y,1)+\eta)-(f_u(x,y,1)+\mu_1-\eta)\psi_1 \rho e^{-\eta t}\le 0.$$
If $(t,x,y)\in \Omega^-$ such that $s_1\le \underline{s}(t,x)\le s_2$, it follows from \eqref{ck}, \eqref{rho} and \eqref{comega} that
$$\mathscr{L}\underline{u}\le -k\omega\rho\eta e^{-\eta t} +M\rho\|g\|_{\infty} e^{-\eta t} +\|B\|_{\infty} \rho e^{-\eta t} +\omega \rho^2 \eta e^{-2\eta t} \|C\|_{\infty} \le 0.$$

In conclusion, $\mathscr{L}\underline{u}\le 0$ for $(t,x,y)\in [0,+\infty)\times\overline{\Omega}$ and $\underline{u}(t,x,y)$ is a subsolution. By the comparison principle, one gets that
$$u(t,x,y)\ge \underline{u}(t,x,y),\hbox{ for $(t,x,y)\in [0,+\infty)\times\overline{\Omega}$}.$$

{\it Step~3: $\overline{u}(t,x,y)$ is a supersolution.}
One can easily check that 
$$\nu A\nabla \overline{u}=0 \hbox{ on $(t,x,y)\in [0,+\infty)\times\partial\Omega$}.$$
Define
$$\Omega^+(t,x,y):=\{(t,x,y)\in [0,+\infty)\times\overline\Omega; \overline{u}(t,x,y)<1\}.$$
After a lengthy but straightforward calculation, one can get that
\begin{align*}
\mathscr{L}\overline{u}:=&\overline{u}_t-\nabla\cdot (A\nabla \overline{u}) +q \cdot \nabla \overline{u} -f(x,y,\overline{u})\\
=& -\omega\rho\eta e^{-\eta t} \phi_s +f(x,y,\phi)-f(x,y,\overline{u}) +B(\overline{s},x,y)\rho e^{-\eta t} -C(\overline{s},x,y)\omega\rho^2\eta e^{-2\eta t},
\end{align*}
for $(t,x,y)\in \Omega^+$ where $\phi_s$, $\phi$ are taken values at $(\overline{s},x,y)$, $\underline{u}$ is taken values at $(t,x,y)$ and $\overline{s}$ denotes $\overline{s}(t,x)$.
If $(t,x,y)\in \Omega^+$ such that $\overline{s}(t,x)\ge s_2$, it follows from \eqref{s2} and \eqref{rho} that
$$\phi(\overline{s},x,y)\le \overline{u}(t,x,y)\le \phi(\overline{s},x,y)+\rho g(\overline{s},x,y)\le \gamma.$$
Then, by \eqref{gamma} and \eqref{s2}, one has that
$$f(x,y,\phi)-f(x,y,\overline{u})\ge -\rho e^{-\lambda \overline{s}} \varphi_{\lambda} e^{-\eta t} (f_u(x,y,0)+\eta).$$
Thus, it follows from $\phi_s<0$ and \eqref{s2} that
\begin{align*}
\mathscr{L}\overline{u}\ge& -\rho e^{-\lambda \overline{s}} \varphi_{\lambda} e^{-\eta t} (f_u(x,y,0)+\eta) +(f_u(x,y,0)+\eta) \rho e^{-\lambda \overline{s}} \varphi_{\lambda} e^{-\eta t}=0.
\end{align*}
If $(t,x,y)\in \Omega^+$ such that $s(t,x)\le s_1$, it follows from \eqref{s1} that
$$\overline{u}(t,x,y)\ge \phi(\overline{s},x,y)\ge 1-\frac{\gamma}{2}.$$
Then, by \eqref{gamma} and \eqref{s1}, one has that
$$f(x,y,\phi)-f(x,y,\overline{u})\ge -\rho \psi_1(x,y) e^{-\eta t}(f_u(x,y,1)+\eta).$$
Thus, it follows from $\phi_s<$0 and \eqref{s1} that
$$\mathscr{L}\overline{u}\ge -\rho \psi_1(x,y) e^{-\eta t}(f_u(x,y,1)+\eta)+(f_u(x,y,1)+\mu_1-\eta)\psi_1 \rho e^{-\eta t}\ge 0.$$
If $(t,x,y)\in \Omega^+$ such that $s_1\le s(t,x)\le s_2$, it follows from  \eqref{ck}, \eqref{rho} and \eqref{comega} that
$$\mathscr{L}\overline{u}\ge k\omega\rho\eta e^{-\eta t} -M\rho\|g\|_{\infty} e^{-\eta t} -\|B\|_{\infty} \rho e^{-\eta t} -\omega \rho^2 \eta e^{-2\eta t} \|C\|_{\infty} \ge 0.$$

In conclusion, $\mathscr{L}\overline{u}\le 0$ for $(t,x,y)\in [0,+\infty)\times\overline{\Omega}$ and $\overline{u}(t,x,y)$ is a supersolution. By the comparison principle, one gets that
$$u(t,x,y)\le \overline{u}(t,x,y),\hbox{ for $(t,x,y)\in [0,+\infty)\times\overline{\Omega}$}.$$

This completes the proof of Lemma~\ref{subsup-s}.
\end{proof}
\vskip 0.3cm

Now, we are ready to prove Theorem~\ref{Th3}.

\begin{proof}[Proof of Theorem~\ref{Th3}]
For $n\in\mathbb{N}$, let 
$$t_n=\frac{nL_1}{c^*}$$ 
where $L_1$ is the period of $x$. Define $u_n(t,x,y)=u(t+t_n,x+nL_1,y)$ for $t\ge -t_n$ and $(x,y)\in \overline{\Omega} -n L_1=\overline{\Omega}$. By standard parabolic estimates, the sequence $u_n(t,x,y)$ converges, up to extraction of a subsequence, to a solution $u_{\infty}(t,x,y)$ of \eqref{RD} locally uniformly in $\R\times\overline{\Omega}$ as $n\rightarrow +\infty$. On the other hand, by Lemma~\ref{subsup-s} and $\phi(s,x,y)$, $g(s,x,y)$ are periodic in $x$, it follows that
\be\label{un}
\begin{aligned}
\phi(\underline{s}_n(t,x),x,y)-&\rho g(\underline{s}_n(t,x)+s_0,x,y)e^{-\eta (t+t_n)}\\
\le& u_n(t,x,y)\le \phi(\overline{s}_n(t,x),x,y)+\rho g(\overline{s}_n(t,x),x,y)e^{-\eta (t+t_n)}, \hbox{ for all $n\in\mathbb{N}$},
\end{aligned}
\ee
where 
$$\underline{s}_n(t,x)=x-c^* t+\omega\rho -\omega\rho e^{-\eta(t+t_n)}+\sigma_0 \hbox{ and }
\overline{s}_n(t,x)=x-c^* t-\omega\rho +\omega\rho e^{-\eta(t+t_n)}-\sigma_0.$$
By passing $n$ to $+\infty$, one has that
$$\phi(x-c^* t+\omega\rho +\sigma_0,x,y)\le u_{\infty}(t,x,y)\le \phi(x-c^* t-\omega\rho -\sigma_0,x,y).$$
By Proposition~\ref{Lio}, there is $\tau$ such that
$$u_{\infty}(t,x,y)=\phi(x-c^* t+\tau,x,y).$$
Together with \eqref{un}, one can get that
\be\label{utn}
u(t_n,x,y)-\phi(x-ct_n+\tau,x,y)\rightarrow 0 \hbox{ as $n\rightarrow +\infty$ uniformly in $\overline{\Omega}$}.
\ee
Take any small constant $\varepsilon>0$. Define
$$\underline{u}(t,x,y)=\max\{\phi(\underline{s}_n(t,x),x,y)-\varepsilon g(\underline{s}_n(t,x)+s_0,x,y) e^{-\eta (t-t_n)},0\},$$
and
$$\overline{u}(t,x,y)=\min\{\phi(\overline{s}_n(t,x),x,y)+\varepsilon g(\overline{s}_n(t,x),x,y) e^{-\eta (t-t_n)},1\},$$
where $\underline{s}(t,x)=x -c^*t+\tau+\varepsilon +\omega\varepsilon -\omega\varepsilon e^{-\eta (t-t_n)}$ and  $\overline{s}(t,x)=x -c^*t +\tau-\varepsilon-\omega\varepsilon +\omega\varepsilon e^{-\eta (t-t_n)}$.
By properties of $g$, one can easily check that for any $s_*\in\R$, there is $\delta_*>0$ such that
$$\phi(s,x,y)-(\phi(s+\varepsilon,x,y)-\varepsilon g(s+\varepsilon+s_0,x,y))\ge \delta_*>0, \hbox{ for $s\le s_*$ and $(x,y)\in\overline{\Omega}$},$$
and
$$\phi(s+\varepsilon,x,y)-g(s+\varepsilon+s_0,x,y)\le 0, \hbox{ for $s\ge s_*$ and $(x,y)\in\overline{\Omega}$}.$$
By \eqref{utn}, there is $N_1>0$ such that for $n\ge N_1$,
\be\label{N1}
\underline{u}(t_n,x,y)\le u(t_n,x,y), \hbox{ for $(x,y)\in\overline{\Omega}$}.
\ee
On the other hand, one can easily check that for any $s'_*\in\R$, there is $\delta'_*>0$ such that
\be\label{phi+eg}
\phi(s-\varepsilon,x,y)+\varepsilon g(s-\varepsilon,x,y)-\phi(s,x,y)\ge \delta_*>0, \hbox{ for $s\le s'_*$ and $(x,y)\in\overline{\Omega}$}.
\ee
Moreover, by \eqref{un}, there are $C_1$, $C_2>0$ such that
$$u(t_n,x,y)\le C_1 e^{-\lambda_*^+(x-c^*t_n)} +C_2 e^{-\eta t_n} e^{-\lambda (x-c^* t_n)}, \hbox{ as $x-c^* t_n\rightarrow +\infty$}.$$
This implies that
$$\frac{\phi(x-c^* t_n+\tau-\varepsilon,x,y)+\varepsilon g(x-c^* t_n+\tau-\varepsilon,x,y)}{u(t_n,x,y)}>1, \hbox{ as $x-c^* t_n\rightarrow +\infty$ for all large $n$},$$
since $g(s,x,y)=e^{-\lambda s}\varphi_{\lambda}(x,y)$ for $s$ large enough and $\lambda<\lambda_*^+$. By \eqref{utn} and \eqref{phi+eg}, there is $N_2>0$ such that for $n\ge N_2$,
\be\label{N2}
u(t_n,x,y)\le \overline{u}(t_n,x,y), \hbox{ for $(x,y)\in\overline{\Omega}$}.
\ee
Thus, by \eqref{N1} and \eqref{N2}, there is $t_N>0$ such that 
\begin{align*}
\underline{u}(t_N,x,y)\le u(t_N,x,y)\le \overline{u}(t_N,x,y), \hbox{ for $(x,y)\in\overline{\Omega}$}.
\end{align*}
By following the same proof of Lemma~\ref{subsup-s}, one can get that
\begin{align*}
\underline{u}(t,x,y)\le u(t,x,y)\le \overline{u}(t,x,y), \hbox{ for $t\ge t_N$ and $(x,y)\in \overline{\Omega}$}.
\end{align*}
As $t\rightarrow +\infty$, one has that
$$\phi(x-c^* t +\tau+\varepsilon+\omega\varepsilon,x,y)\le u(t,x,y)\le \phi(x-c^* t +\tau-\varepsilon-\omega\varepsilon,x,y).$$
Since $\varepsilon$ is arbitrarily small and $|\phi_s|$ is globally bounded, one finally has that
$$u(t,x,y)\rightarrow \phi(x-c^* t +\tau,x,y) \hbox{ as $t\rightarrow +\infty$ uniformly in $\overline{\Omega}$}.$$
This completes the proof.
\end{proof}


\end{document}